\documentclass[12pt,reqno]{amsart}

\usepackage{amscd,amssymb}
\usepackage{epsfig}
\usepackage{float}
\usepackage{url}
\usepackage[all]{xy}
\usepackage{tikz}
\newcommand{\N}{{\mathbb N}}
\newcommand{\Z}{{\mathbb Z}}
\newcommand{\Q}{{\mathbb Q}}
\newcommand{\C}{{\mathbb C}}

\newcommand{\F}{{\mathbb F}}
\renewcommand{\P}{{\mathbb P}}

\newcommand{\BB}{{\mathcal B}}

\newcommand{\EE}{{\mathcal E}}

\newcommand{\LL}{{\mathcal L}}

\newcommand{\OO}{{\mathcal O}}

\newcommand{\www}{\widetilde}

\newcommand{\oooo}{\overline}

\DeclareMathOperator{\Ann}{Ann}

\DeclareMathOperator{\Hom}{Hom}

\DeclareMathOperator{\mmod}{mod}
\DeclareMathOperator{\pr}{pr}

\DeclareMathOperator{\rank}{rank}

\begin{document}

\theoremstyle{plain}
\newtheorem{lemma}{Lemma}[section]
\newtheorem{definition/lemma}[lemma]{Definition/Lemma}
\newtheorem{theorem}[lemma]{Theorem}
\newtheorem{proposition}[lemma]{Proposition}
\newtheorem{corollary}[lemma]{Corollary}
\newtheorem{conjecture}[lemma]{Conjecture}
\newtheorem{conjectures}[lemma]{Conjectures}
\newtheorem{question}[lemma]{Question}

\theoremstyle{definition}
\newtheorem{definition}[lemma]{Definition}
\newtheorem{withouttitle}[lemma]{}
\newtheorem{remark}[lemma]{Remark}
\newtheorem{remarks}[lemma]{Remarks}
\newtheorem{example}[lemma]{Example}
\newtheorem{examples}[lemma]{Examples}
\newtheorem{notations}[lemma]{Notations}
\newtheorem{problem}[lemma]{Problem}

\title[Semigroups from full lattices in $\Q$-algebras]
{Semigroups from full lattices
in commutative $\Q$-algebras} 

\author{Claus Hertling and Khadija Larabi}

\address{Claus Hertling\\
Lehrstuhl f\"ur algebraische Geometrie, 
Universit\"at Mannheim,
B6 26, 68159 Mannheim, Germany}

\email{hertling@math.uni-mannheim.de}

\email{khadija.larabi@outlook.com}

\date{February 16, 2026}




\begin{abstract}
{\Small The full lattices in a finite dimensional commutative
$\Q$-algebra form a commutative semigroup. In the case of an
algebraic number field the top part of a certain
quotient semigroup is the class group.
For a separable algebra some basic results, especially the
Jordan-Zassenhaus theorem, are known for this quotient semigroup.
This paper considers also algebras which are not separable. 
It studies the commutative semigroup of full lattices
in such an algebra and also the quotient semigroup.
This leads in this commutative, but not separable 
situation to a certain extension of
the Jordan-Zassenhaus theorem. One application
concerns $GL_n(\Z)$-conjugacy classes of regular integer 
$n\times n$ matrices.}
\end{abstract}

\maketitle

\tableofcontents

\setcounter{section}{0}

\section{Introduction}\label{s1}
\setcounter{equation}{0}
\setcounter{table}{0}

One motivation for the results in this paper is the following
problem.

\begin{problem}\label{t1.1}
Let $n\in\Z_{\geq 2}$, and let $J\in M_{n\times n}(\C)$ be a
matrix in Jordan normal form which has for each eigenvalue
only one Jordan block and whose characteristic polynomial
$f$ is in $\Z[t]$. Study the set $S_{1,f}$ of 
$GL_n(\Z)$-conjugacy classes of matrices in $M_{n\times n}(\Z)$
which are $GL_n(\C)$-conjugate to $J$, and appreciate all 
structure which this set has.
\end{problem}

The set $S_{1,f}$ is not empty, because it contains the 
$GL_n(\Z)$-conjugacy class of the companion matrix
$M_f:=\begin{pmatrix} & & & -f_0 \\ 1  & & & -f_1 \\
 & \ddots & & \vdots \\ & & 1 & -f_{n-1}\end{pmatrix}$
(with zeros at empty places) of the polynomial
$f=t^n+f_{n-1}t^{n-1}+...+f_1t+f_0\in\Z[t]$. 
In fact, each matrix in $M_{n\times n}(\Z)$, which is 
$GL_n(\C)$-conjugate to $J$, is $GL_n(\Q)$-conjugate to $M_f$,
but not necessarily $GL_n(\Z)$-conjugate to $M_f$. 

By the Jordan-Zassenhaus theorem \cite{Za38} the set $S_{1,f}$
is finite if $J$ is semisimple, so if $f$ has simple roots.
Else it is infinite.

One result of this paper is a finiteness statement also for the
case when $J$ is not semisimple and $S_{1,f}$ is infinite.
It can be seen as an extension of the Jordan-Zassenhaus theorem
to this situation. See below Theorem 1.3 (a).

In any case, the set $S_{1,f}$ is a commutative semigroup
and comes equipped additionally with a division map.
We are interested in these structures. At the end of this
introduction we come back to $S_{1,f}$ and explain this.

Now we introduce the main objects of this paper.
Throughout the whole paper $A$ will be a finite dimensional
commutative $\Q$-algebra with unit element $1_A$. The structure of
$A$ is not difficult. It splits into a direct sum $A=F\oplus R$
with $F$ a sum of algebraic number fields and $R$ the radical,
which is the ideal of all nilpotent elements.
$A$ is separable if and only if $A=F$, so $R=0$. Define
\begin{eqnarray*}
\LL(A):= \{L\subset A\,|\, L\textup{ is a }
\Z\textup{-lattice which generates }A\textup{ over }\Q\}.
\end{eqnarray*}
The elements of $A$ are called {\it full lattices}.
The multiplication in $A$ induces a multiplication and a 
division map on $\LL(A)$. If $L_1$ and $L_2$ are two full 
lattices then
\begin{eqnarray*}
L_1\cdot L_2&:=&\{\sum_{i\in I}a_ib_i\,|\, I\textup{ a finite 
index set}, a_i\in L_1,b_i\in L_2\}\\
\textup{and}\quad L_1:L_2&:=& \{a\in A\,|\, aL_2\subset L_1\}
\end{eqnarray*}
are also full lattices. $\LL(A)$ with this multiplication
is a commutative semigroup. A full lattice $\Lambda$ in $A$
is an {\it order} if
\begin{eqnarray*}
1_A\in\Lambda\quad\textup{and}\quad \Lambda\cdot\Lambda
\subset\Lambda\quad(\textup{then }\Lambda\cdot\Lambda=\Lambda).
\end{eqnarray*}
For each full lattice $L$ in $A$
$$\OO(L):=L:L$$
is an order and is called {\it the order of }$L$.

Though $\LL(A)$ is (too) big. A natural equivalence relation on
$\LL(A)$ is the following $\varepsilon$-equivalence,
\begin{eqnarray*}
L_1\sim_\varepsilon L_2\ \Leftrightarrow_{\textup{Def.}}\ 
\exists\ a\in A^{unit}\textup{ with }aL_1=L_2,\\
\textup{where}\quad A^{unit}:=(\textup{the group of units in }A).
\end{eqnarray*}
Multiplication and division map on $\LL(A)$ descend to the
quotient
$$\EE(A):=\LL(A)/\sim_\varepsilon,$$
so especially also $\EE(A)$ is a commutative semigroup.
Our results are on the semigroups $\LL(A)$ and $\EE(A)$.

In order to formulate them we have to say a few words about
the structure of an arbitrary commutative semigroup $S$.
An element $c\in S$ with $cc=c$ is called an {\it idempotent}.
An element $a\in S$ is called {\it invertible} if an idempotent
$c$ and an element $b\in S$ with $ac=a$ and $ab=c$ exist.
Then $c$ is unique and is called $e_a$, and a unique
invertible element $a^{-1}\in S$ with $aa^{-1}=e_a$, 
$a^{-1}e_a=a^{-1}$ and $e_{a^{-1}}=e_a$ exists.
Then for each idempotent $c$ the set
$$G(c):=\{a\in S\,|\, a\textup{ invertible with }e_a=c\}$$
is a group with unit element $c$. The subset
$\bigcup_{c\textup{ idempotent}}G(c)$ of $S$ of all invertible
elements is a subsemigroup of $S$. 
An natural equivalence relation on $S$ is the following
$w$-equivalence,
\begin{eqnarray*}
a\sim_w b &\Leftrightarrow_{\textup{Def.}}& a=b\quad\textup{or}\\
&& \exists\ x_1,x_2\in S\textup{ with }ax_1=b,bx_2=a.
\end{eqnarray*}
Then $W(S):=S/\sim_w$ is a quotient semigroup, and for
an idempotent $c\in S$ $G(c)=[c]_w$.

\begin{lemma}\label{t1.2}
(a) The idempotents in $\LL(A)$ are the orders.
The idempotents in $\EE(A)$ are the $\varepsilon$-classes
of the orders.

(b) $L$ is invertible in $\LL(A)$ if and only if 
$[L]_\varepsilon$ is invertible in $\EE(A)$.
Then $e_L=\OO(L)$ and $e_{[L]_\varepsilon}=[\OO(L)]_\varepsilon$.

(c) $W(\LL(A))=W(\EE(A))$. Also the division map on $\LL(A)$
descends to a division map on $W(\LL(A))$.
\end{lemma}

The following theorem collects the five main results of this
paper. After the theorem we comment upon them.

\begin{theorem}\label{t1.3}
(a) (Theorem \ref{t6.3}) For each order $\Lambda$ the set
$$\{[L]_\varepsilon\,|\, L\in\LL(A),\OO(L)=\Lambda\}\subset
\EE(A)$$
is finite.

(b) (Theorem \ref{t5.8} (b)) For each order $\Lambda$ the
set $\{[L]_\varepsilon\,|\, L\in\LL(A),\OO(L)=\Lambda\}$
splits into finitely many $w$-classes, which have all the same
(finite by part (a)) size. One of them is the group
$G([\Lambda]_\varepsilon)$.

(c) (Part of Theorem \ref{t8.2}) Let $\Lambda_1$ and $\Lambda_2$
be orders with $\Lambda_2\subset\Lambda_1$. There are natural
surjective group homomorphisms
$G(\Lambda_2)\to G(\Lambda_1)$, $L\mapsto\Lambda_1L$, and
$G([\Lambda_2]_\varepsilon)\to G([\Lambda_1]_\varepsilon)$,
$[L]_\varepsilon\mapsto [\Lambda_1L]_\varepsilon$.

(d) (Theorem \ref{t9.2}) The projection $A\to F$ with respect
to the splitting $A=F\oplus R$ induces surjective 
homomorphisms $\LL(A)\to \LL(F)$ and $\EE(A)\to\EE(F)$
of semigroups. For each order $\Lambda$, they restrict to a 
surjective group homomorphism $G(\Lambda)\to G(\pr_F\Lambda)$
and to an isomorphism 
$$G([\Lambda]_\varepsilon)\to G([\pr_F\Lambda]_\varepsilon)$$ 
of finite groups.

(e) (Theorem \ref{t10.1}) Suppose $\dim A\geq 2$.
For each full lattice $L$ in $A$ each power $L^k$ with
$k\geq \dim A-1$ is invertible.
\end{theorem}

Special cases of these results are old and were proved mainly
in the 1960ies in a paper \cite{DTZ62} of Dade, Taussky and
Zassenhaus, in several papers \cite{Fa67}\cite{Fa68}
of Faddeev, and in the book \cite{Ne99} of Neukirch.

Part (a) for $A$ separable follows from the Jordan-Zassenhaus
theorem \cite{Za38}. Therefore it can be considered as an
extension of the Jordan-Zassenhaus theorem. For $A$ with
radical $R\neq 0$ with $R^2=0$ it is proved in \cite{Fa67}.
Part (b) is new.
Part (c) (and all of Theorem \ref{t8.2}) for an algebraic 
number field $A$ is proved in \cite{Ne99}.
Part (d) is new except for the injectivity of the group 
homomorphism $G([\Lambda]_\varepsilon)\to 
G([\pr_F\Lambda]_\varepsilon)$ which is proved in \cite{Fa68}.
Part (e) for $A$ an algebraic number field is the main result
of \cite{DTZ62}.

The five parts of Theorem \ref{t1.3} are proved in the sections
\ref{s5} (at the end of it), \ref{s6}, \ref{s8}, \ref{s9} and 
\ref{s10}. The sections \ref{s2}, \ref{s3}, \ref{s4}, \ref{s5}
and \ref{s7} give background material.
The sections \ref{s2} and \ref{s3} are rather elementary. 
The sections \ref{s4} and \ref{s5} (except for Theorem 
\ref{t5.8}) are condensed from \cite{DTZ62} 
(but there the assumptions are more restrictive),
the first half of section \ref{s7} is from \cite{Fa65}.

\cite{DTZ62} and \cite{Ne99} consider only an algebraic 
number field $A$ and work with localizations by prime ideals
of orders in it. In section \ref{s7} we follow \cite{Fa65}
and introduce another family of localizations which works
for arbitrary $A$. It is used in the proofs of the parts
(c), (d) and (e) in the sections \ref{s8}, \ref{s9} and 
\ref{s10}.

The five main results in Theorem \ref{t1.3} are supported by
the two more technical Theorems \ref{t7.6} and \ref{t10.2}.

Two points which deserve further study are (1) the division 
maps in $\LL(A)$ and $\EE(A)$ and (2) the properties of
not invertible elements in the semigroups $\LL(A)$ and $\EE(A)$.
With respect to (2) we only have Theorem \ref{t1.3} (e).

Finally we come back to the set $S_{1,f}$ where $f\in\Z[t]$ is
unitary of some degree $n\in\Z_{\geq 2}$. Consider
$$A_f:=\frac{\Q[t]}{(f)_{\Q[t]}}\supset\Lambda_f :=
\frac{\Z[t]}{(f)_{\Z[t]}}.$$
$A_f$ is an $n$-dimensional commutative $\Q$-algebra with
unit element $1_A$, $\Lambda_f$ is an order in $A_f$. Define the set
$$S_{2,f}:=\{[L]_\varepsilon\,|\, L\in\LL(A),\OO(L)\supset
\Lambda_f\}.$$
It is a subsemigroup of $\EE(A)$, and also the division map on
$\EE(A_f)$ restricts to it. The following old observation is
elementary, but crucial.

\begin{lemma}\label{t1.4}\cite{LMD33}
There is a natural 1:1-correspondence between the sets 
$S_{1,f}$ and $S_{2,f}$. 

From an $\varepsilon$-class in $S_{2,f}$ to a conjugacy class
in $S_{1,f}$: Choose $L$ in the $\varepsilon$-class in $S_{2,f}$.
Choose a $\Z$-basis $(b_1,...,b_n)$ of $L$. Then
\begin{eqnarray*}
(\oooo{t}b_1,...,\oooo{t}b_n)
=(b_1,...,b_n)\cdot B\quad\textup{for some matrix}\quad
B\in M_{n\times n}(\Z).
\end{eqnarray*}
The $GL_n(\Z)$-conjugacy class of $B$ is in $S_{1,f}$.
\end{lemma}

Therefore also $S_{1,f}$ is a commutative semigroup with a 
division map. We have no direct applications of these structures
on $S_{1,f}$, but we find them fascinating.

This paper does not contain examples. But it is accompanied by
another longer paper \cite{HL26} which contains a lot of 
examples: All cases with $n=2$, many cases with $n=3$ and
some cases with arbitrary $n$.

We finish the introduction with some notations.

\begin{notations}\label{t1.5}
(a) Throughout the whole paper $A$ is a finite dimensional 
commutative $\Q$-algebra with unit element $1_A$.

(b) The annihilator $\Ann(B)$ of a subset $B\subset A$
is the $\Q$-vector space $\Ann(B):=\{a\in A\,|\, ab=0
\textup{ for }b\in B\}.$

(c) $\N=\{1,2,...\}$,
$\P:=\{\textup{prime numbers in }\N\}$,\\
$\F_p=\Z/p\Z$ the field with $p$ elements for $p\in\P$.

(d) For a commutative ring $R$ with unit element $1_R$, the group of its units
is called $R^{unit}:=\{a\in R\,|\, \exists\ b\in R\textup{ with }ab=1_R\}$.
\end{notations}

\section{Full lattices in finite dimensional $\Q$-vector spaces}
\label{s2}
\setcounter{equation}{0}
\setcounter{table}{0}

This section starts with a definition and then presents several
elementary lemmas which will be useful later.

\begin{definition}\label{t2.1}
Let $V$ be finite dimensional $\Q$-vector space.

(a) A {\it lattice in $V$} is a finitely generated
$\Z$-module $L\subset V$. 
A {\it full lattice in $V$} is a lattice in $V$ which generates
$V$ over $\Q$. 
The set of all full lattices in $V$ is called $\LL(V)$.

(b) The dual space of $V$ is denoted $V^*:=\Hom_\Q(V,\Q)$. 
For each subset $M\subset V$ define 
$M^{*\Z}:=\{C\in V^*\,|\, C(m)\in\Z\textup{ for all }m\in M\}
\subset V^*$. It is a subgroup of $V^*$ as an additive group. 
\end{definition}

The following lemma puts together some trivial or well known
facts.

\begin{lemma}\label{t2.2}
Let $V$ be a finite dimensional $\Q$-vector space.

(a) A lattice $L$ in $V$ does not have torsion. Therefore it has
a $\Z$-basis and a rank $\rank L\in\Z_{\geq 0}$, 
which is the number of elements of a $\Z$-basis of $L$. 
The rank of $L$ is also the dimension
of the $\Q$-subspace $\Q\cdot L\subset V$ which is generated
by $L$. A lattice $L$ is a full lattice if and only if
$\Q\cdot L=V$.

(b) $(V^*)^*=V$. If $L$ is a lattice in $V$ then 
$(L^{*\Z})^{*\Z}=L$. 
Furthermore then $L^{*\Z}$ is a full lattice
in $V^*$ if and only if $L$ is a full lattice in $V$. 

(c) Let $L_1$ and $L_2$ be full lattices in $V$. Then also
$L_1+L_2$ and $L_1\cap L_2$ are full lattices in $V$.
There are $k_1$ and $k_2\in\N$ with 
$$k_1L_1\subset L_2\subset k_2^{-1}L_1.$$

(d) Let $L_1$ be a full lattice in $V$ and $N\in\N$. The sets
$\{L\in\LL(V)\,|\, L\supset L_1,[L:L_1]\leq N\}$ and 
$\{L\in \LL(V)\,|\, L\subset L_1,[L_1:L]\leq N\}$ are finite.
Especially, if $L_2$ is a full lattice with $L_2\supset L_1$,
then the set $\{L\in\LL(V)\,|\, L_1\subset L\subset L_2\}$
is finite. 

(e) Let $L\subset V$ be a $\Z$-module which is not a lattice.
Then $0\in V$ is an accumulation point of $L$.
\end{lemma}

The next lemma will be useful again and again.

\begin{lemma}\label{t2.3}
Let $V_1,V_2$ and $V_3$ be finite dimensional $\Q$-vector spaces.
Let $L_1\in\LL(V_1)$ and $L_3\in\LL(V_3)$, and let 
$\beta:V_1\times V_2\to V_3$ be a $\Q$-bilinear map such that
for some elements $a_1,...,a_l\in V_1$
\begin{eqnarray}\label{2.1}
\bigcap_{i=1}^l \ker\Bigl(\beta(a_i,.):V_2\to V_3\Bigr)=\{0\}.
\end{eqnarray}

(a) Then the set
\begin{eqnarray*}
L_3:L_1&:=&\{ b\in V_2\,|\, \beta(.,b)(L_1)\subset L_3\}
\end{eqnarray*}
is a full lattice in $V_2$, so $L_3:L_1\in\LL(V_2)$.

(b) A dual bilinear map $\beta^*$ is defined by
\begin{eqnarray*}
\beta^*: V_1\times V_3^*&\to& V_2^*,\\
\beta^*(a,C)(b)&:=& C(\beta(a,b))\quad\textup{for }
a\in V_1,C\in V_3^*,b\in V_2.
\end{eqnarray*}
Then
\begin{eqnarray*}
(L_3:L_1)^{*\Z} &=& \beta^*(L_1,L_3^{*\Z}), \\
\textup{ equivalently: }
L_3:L_1 &=&(\beta^*(L_1,L_3^{*\Z}))^{*\Z}.
\end{eqnarray*}
\end{lemma}

{\bf Proof:}
(a) The set $L_3:L_1$ is a $\Z$-module. Suppose that it is not
a lattice in $V$. Then by Lemma \ref{t2.2} (e) there is a 
sequence $(b_n)_{n\in\N}$ of points in $L_3:L_1-\{0\}$ with 
$\lim_{n\to\infty}b_n=0$. 

For each $a_i$ in \eqref{2.1}, the sequence 
$(\beta(a_i,b_n))_{n\in\N}$ consists of points in $L_3$ 
and converges to $0$. So there is an $n_i\in\N$ with 
$\beta(a_i,b_n)=0$ for $n\geq n_i$. Then $b_n=0$ for
$n\geq \max_in_i$ by the assumption \eqref{2.1}, 
a contradiction. Therefore $L_3:L_1$ is a lattice.

It is a full lattice in $V_2$ because for any $b\in V_2$
a number $r\in \N$ with $\beta(.,rb)(L_1)\subset L_3$ exists.

(b) We have $(\beta^*(L_1,L_3^{*\Z}))^{*\Z}=L_3:L_1$ because of the 
following.
\begin{eqnarray*}
&& b\in (\beta^*(L_1,L_3^{*\Z}))^{*\Z}\subset (V^*_2)^*=V_2\\
&\iff& C(\beta(a,b))=\beta^*(a,C)(b)\in\Z\quad
\textup{for all }a\in L_1,C\in L_3^{*\Z}\\
&\iff& \beta(a,b)\in L_3\quad\textup{for all }a\in L_1\\
&\iff& b\in L_3:L_1.
\end{eqnarray*}

$\beta^*(L_1,L_3^{*\Z})$ is a lattice in $V_2^*$ because $L_1$ and
$L_3^{*\Z}$ are lattices in $V_1$ and $V_3^*$. 
As $(\beta^*(L_1,L_3^{*\Z}))^{*\Z}=L_3:L_1$ is a full lattice in $V_2$,
by Lemma \ref{t2.2} (b) $\beta^*(L_1,L_3^{*\Z})$ is a full lattice
in $V_2^*$ and $\beta^*(L_1,L_3^{*\Z})=(L_3:L_1)^{*\Z}$.
\hfill$\Box$ 

\begin{lemma}\label{t2.4}
Let $V$ be a finite dimensional $\Q$-vector space with an 
increasing filtration
\begin{eqnarray*}
\{0\}=V_0\subset V_1\subset...\subset V_{n-1}\subset V_n=V
\end{eqnarray*}
for some $n\in\N$. Denote
$$V_{[j]}:= V_j/V_{j-1}\quad\textup{for }j\in\{1,...,n\}$$
and denote for $L\in \LL(V)$
$$L_{[j]}:= (L\cap V_j+V_{j-1})/V_{j-1}\subset V_{[j]}
\quad\textup{for }j\in\{1,...,n\}.$$

(a) Then $L_{[j]}\in\LL(V_{[j]})$.

(b) If $K$ and $L\in \LL(V)$ satisfy $K\supset L$ then
\begin{eqnarray}\label{2.2}
K_{[j]}\supset L_{[j]}\quad\textup{and}\quad
[K:L]=\prod_{j=1}^n [K_{[j]}:L_{[j]}].
\end{eqnarray}

(c) Fix $L\in\LL(V)$. Consider for any $j\in\{1,...,n\}$
a finite set $\{K_{j,1},...,K_{j,l_j}\}\subset\LL(V_{[j]})$.
Then the set
\begin{eqnarray*}
\{K\in\LL(V)\,|\, K\supset L, 
K_{[j]}\in \{K_{j,1},...,K_{j,l_j}\}\textup{ for }
j\in\{1,...,n\}\}
\end{eqnarray*}
is finite.
\end{lemma}

{\bf Proof:}
(a) Trivial.

(b) We restrict to the case $n=2$. The general case follows
similarly respectively by induction. For $j\in\{1;2\}$
denote $m_j:=\dim V_{[j]}$. For $j\in\{1;2\}$ choose 
$\Z$-bases $\BB_{K,j}\in M_{1\times m_j}(K_{[j]})$ and 
$\BB_{L,j}\in M_{1\times m_j}(L_{[j]})$ of $K_{[j]}$ 
respectively $L_{[j]}$, 
and denote by $M_j\in M_{m_j\times m_j}(\Z)$ 
the matrix which expresses the elements of $\BB_{L,j}$
as linear combinations of the elements of $\BB_{K,j}$,
namely $\BB_{L,j}=\BB_{K,j}\cdot M_j$. Then
$[K_{[j]}:L_{[j]}]=|\det M_j|$. 

Lift the elements of $\BB_{K,2}$ and of $\BB_{L,2}$
arbitrarily to elements of $K$ respectively $L$
and call the corresponding tuples $\www{\BB}_{K,2}$
and $\www{\BB}_{L,2}$. 
Then $(\www{\BB}_{K,2},\BB_{K,1})$ is a $\Z$-basis of $K$,
and $(\www{\BB}_{L,2},\BB_{L,1})$ is a $\Z$-basis of $L$,
and the matrix which expresses the elements of this 
$\Z$-basis of $L$ as linear combinations of this $\Z$-basis
of $K$ is
$$\begin{pmatrix}M_2 & 0 \\ * & M_1 \end{pmatrix}
\in M_{(m_1+m_2)\times (m_1+m_2)}(\Z).$$
The absolute value of its determinant is
$$[K:L]=|\det M_2|\cdot |\det M_1| =[K_{[2]}:L_{[2]}]
\cdot [K_{[1]}:L_{[1]}].$$

(c) A full lattce $K\supset L$ with
$K_{[j]}\in \{K_{j,1},...,K_{j,l_j}\}$ satisfies because 
of part (b)
$$[K:L]\leq \prod_{j=1}^n \max_i [K_{j,i}:L_{[j]}],$$
so the number $[K:L]$ is bounded from above. 
One concludes with Lemma \ref{t2.2} (d)
that the set of such lattices $K$ is finite.
\hfill$\Box$

\section{Finite dimensional commutative $\Q$-algebras}\label{s3}
\setcounter{equation}{0}
\setcounter{table}{0}
\setcounter{figure}{0}

Throughout the whole paper $A$ is a finite dimensional
commutative $\Q$-algebra with unit element $1_A$.
The structure theory of $A$ is not difficult.
It is partly covered by the Wedderburn-Malcev theorem
\cite[(72.19) Theorem]{CR62}. But we need to look at the
structure more closely. This is done in Theorem \ref{t3.1}
and Lemma \ref{t3.3}.

\begin{theorem}\label{t3.1}
Let $A$ be a finite dimensional commutative $\Q$-algebra
with unit element $1_A$. Then $A$ has a unique vector space 
decomposition
\begin{eqnarray}
A=\bigoplus_{j=1}^k A^{(j)} \quad\textup{for some }
k\in\N \label{3.1}
\end{eqnarray}
with the following properties.
\begin{list}{}{}
\item[(i)]
$A^{(1)},...,A^{(k)}$ are $\Q$-algebras with 
$A^{(i)}\cdot A^{(j)}=\{0\}$ for $i\neq j$.
\item[(ii)]
$1_A=\sum_{j=1}^k 1_{A^{(j)}}$ splits into unit elements $1_{A^{(j)}}$
of the algebras $A^{(j)}$. 
\item[(iii)]
$A^{(j)}$ splits uniquely into
\begin{eqnarray}\label{3.2}
A^{(j)}=F^{(j)}\oplus N^{(j)}
\end{eqnarray}
with $F^{(j)}$ a $\Q$-subalgebra and an algebraic number field,
and $N^{(j)}$ the unique maximal ideal of $A^{(j)}$.
$A^{(j)}$ is an $F^{(j)}$-algebra. 
$N^{(j)}$ consists of nilpotent elements and is the radical
of $A^{(j)}$.
\end{list}
\end{theorem}

{\bf Proof:}
For each $a\in A$ let $\mu_a:A\to A,\ b\mapsto ab$, denote the
multiplication by $a$. The $\Q$-vector space $A$ decomposes
into hauptspaces with respect to the $\Q$-linear endomorphism
$\mu_a$. As all $\mu_a$ commute, $A$ has the unique simultaneous
hauptspace decomposition $A=\bigoplus_{j=1}^k A^{(j)}$
with respect to all $\mu_a$. In fact, it is sufficient to
consider a $\Q$-basis $a_1,...,a_{\dim A}$ of $A$ and the 
intersections of the hauptspaces of 
$\mu_{a_1},...,\mu_{a_{\dim A}}$.

Each $A^{(j)}$ is invariant under $\mu_a$ for each $a\in A$, so 
$A^{(j)}$ is a $\Q$-algebra with 
$$A^{(i)}\cdot A^{(j)}\subset A^{(i)}\cap A^{(j)}=\{0\}
\quad\textup{for }i\neq j.$$
Decompose $1_A$ into its summands in $A^{(1)},...,A^{(k)}$,
$1_A=\sum_{j=1}^k 1_{A^{(j)}}$. 
For $a\in A^{(j)}$ we have $a=1_A\cdot a =1_{A^{(j)}}\cdot a$, so 
$1_{A^{(j)}}$ is a unit element in $A^{(j)}$.

For each $a\in A$ the Jordan-Chevalley theorem gives a unique
decomposition $\mu_a=(\mu_a)_s+(\mu_a)_n$ with commuting
endomorphisms $(\mu_a)_s$ and $(\mu_a)_n$ in $\Q[\mu_a]$,
where $(\mu_a)_s$ is semisimple and $(\mu_a)_n$ is nilpotent
\cite[VII \S 5 9. Theorem 1]{Bo90}.
Then $a=a_s+a_n$ with $a_s,a_n\in\Q[a]\subset A$,
$\mu_{a_s}=(\mu_a)_s$, $\mu_{a_n}=(\mu_a)_n$. 

The decomposition $a=a_s+a_n$ for $a\in A^{(j)}$ 
gives a decomposition $A^{(j)}=F^{(j)}\oplus N^{(j)}$ 
with $a_s$ semisimple and $a_n$ nilpotent.
$F^{(j)}$ contains elements $a$ such that the characteristic
polynomial of the restriction $\mu_a|_{F^{(j)}}$ is irreducible.
Else $F^{(j)}$ and $A^{(j)}$ would decompose further.
Therefore 
$$F^{(j)}=\Q[a]\cong \Q[t]/(
\textup{characteristic polynomial of }\mu_a|_{F_j})$$ 
is an algebraic number field.
\hfill$\Box$

\begin{remarks}\label{t3.2}
(i) In fact, the Jordan-Chevalley decomposition holds over 
any perfect field \cite[VII \S 5 9. Theorem 1]{Bo90}.
Therefore we can replace in Theorem \ref{t3.1} 
the field $\Q$ by any perfect field $k$, 
especially by a finite field.
We will use this in the proofs of Theorem \ref{t7.6}
and Theorem \ref{t10.3}.

Though then the subspaces $F^{(j)}$ are in general not 
algebraic number fields. Then they are finite extension fields 
of $k$.

(ii) The sum 
\begin{eqnarray}\label{3.3}
F:=\bigoplus_{j=1}^k F^{(k)}
\end{eqnarray}
is a separable subalgebra of $A$. The maximal ideals
in $A$ are the subspaces
$N^{(j)}\oplus\bigoplus_{i\neq j}A^{(i)}$ for $j\in\{1,...,k\}$.
Their intersection is the radical
\begin{eqnarray}\label{3.4}
R=\bigoplus_{j=1}^k N^{(j)}.
\end{eqnarray}
It is the set of all nilpotent elements of $A$. 
The algebra $A$ decomposes naturally 
into the direct sum $A=F\oplus R$. The induced projection
\begin{eqnarray}\label{3.5}
\pr_F:A\to F
\end{eqnarray}
respects addition, multiplication and division. 
We will come back to it in section \ref{s9}.

(iii) We have 
\begin{eqnarray*}
F^{unit}&=&\prod_{j=1}^k(F^{(j)})^{unit}
=\prod_{j=1}^k(F^{(j)}-\{0\}),\\
A^{unit}&=& F^{unit}\times R,\\
\pr_F(A^{unit})&=&F^{unit}\quad\textup{and}\quad 
\pr_F^{-1}(F^{unit})=A^{unit}.
\end{eqnarray*}
\end{remarks}

The powers of $R$ and the annihilators of its powers 
give rise to two filtrations on $A$ which are related 
in a good way, see Lemma \ref{t3.3}.

\begin{lemma}\label{t3.3}
Consider the situation in Theorem \ref{t3.1}.

(a) For $j\in\{1,...,k\}$ denote $n_j:=\max(n\in\Z_{\geq 0}\,|\, 
(N^{(j)})^n\neq\{0\})$, and denote $n_{max}:=\max_j n_j$.
Consider the decreasing filtration of $A$ by powers of $R$,
\begin{eqnarray*}
A=R^0\supset R^1\supset ... \supset R^{n_{max}}\supset
R^{n_{max}+1}=\{0\}
\end{eqnarray*}
(here $j$ in $R^j$ serves simultaneously as upper index and as
exponent). It is called {\sf radical filtration}.
It satisfies 
\begin{eqnarray}\label{3.6}
R^{l_1}\cdot R^{l_2}=R^{l_1+l_2}\quad\textup{for }
l_1,l_2\in\Z_{\geq 0}.
\end{eqnarray}
Denote the quotients of this filtration by
$$R^{[l]}:= R^l/R^{l+1}\quad\textup{for }
l\in\{0,1,...,n_{max}\}.$$
The multiplication on $A$ induces multiplications on pairs of
quotients with
\begin{eqnarray}\label{3.7}
R^{[l_1]}\cdot R^{[l_2]}=R^{[l_1+l_2]}\quad\textup{for }
l_1,l_2\in\Z_{\geq 0}.
\end{eqnarray}
All the structure in part (a) is compatible with the
decomposition $A=\bigoplus_{j=1}^kA^{(j)}$, especially
\begin{eqnarray*}
R^l&=& \bigoplus_{j=1}^k R^l\cap A^{(j)}\quad\textup{with}\quad
R^l\cap A^{(j)}= (N^{(j)})^l,\\
R^{[l]}&\cong&\bigoplus_{j=1}^k (N^{(j)})^l/(N^{(j)})^{l+1}
\quad\textup{canonically.}
\end{eqnarray*}

(b) Consider the increasing filtration of $A$ by the 
annihilators $\Ann(R^l)$,
\begin{eqnarray*}
S_l:= \Ann(R^l)\quad\textup{for }l\in\{0,1,...,n_{max}+1\},\\
\{0\}=S_0\subset S_1\subset ...\subset S_{n_{max}}
\subset S_{n_{max}+1}=A.
\end{eqnarray*}
It is called {\sf socle filtration} because $S_1=\Ann(R)$
is the socle of $A$. It satisfies
\begin{eqnarray}
R^{l_1}\cdot S_{l_2}&\subset& S_{l_2-l_1}\quad\textup{for }
l_1<l_2,\label{3.8}\\
R^{l_1}\cdot S_{l_2}&=& \{0\}\quad\textup{for }
l_1\geq l_2.\nonumber
\end{eqnarray}
Denote the quotients of the socle filtration by
$$S_{[l]}:= S_l/S_{l-1}\quad\textup{for }
l\in\{1,...,n_{max}+1\}.$$
The multiplication on $A$ induces multiplications on pairs
of quotients with
\begin{eqnarray}\label{3.9}
R^{[l_1]}\cdot S_{[l_2]}&\subset& S_{[l_2-l_1]}\quad\textup{for }
l_1<l_2,\\
R^{[l_1]}\cdot S_{[l_2]}&=& \{0\}\quad\textup{for }
l_1\geq l_2.\nonumber
\end{eqnarray}
Consider $l_1,l_2\in\Z_{\geq 0}$ with $l_1<l_2$, and consider
a $\Q$-basis $a_1,...,a_{\dim R^{[l_1]}}$ of $R^{[l_1]}$. 
The homomorphisms
$$\mu^{[l_2]}_{a_i}:S_{[l_2]}\to S_{[l_2-l_1]}, b\mapsto
a_i\cdot b$$
satisfy 
\begin{eqnarray}\label{3.10}
\Bigl(\bigcap_{i=1}^{\dim R^{[l_1]}}\ker \mu^{[l_2]}_{a_i}
\Bigr)=\{0\}\subset S_{[l_2]}.
\end{eqnarray}
All the structure in part (b) is compatible with the
decomposition $A=\bigoplus_{j=1}^kA^{(j)}$.
\end{lemma}

{\bf Proof:}
(a) Equality in \eqref{3.6} follows from the definition of
the power $R^l$ of $R$. It implies \eqref{3.7}.
Everything else is straightforward.

(b) \eqref{3.8} follows from the definition of $S_l$.
It implies \eqref{3.9}. The main point is to prove \eqref{3.10}.

Suppose $b\in S_{[l_2]}$ with $a_i\cdot b=0\in S_{[l_2-l_1]}$
for each $i$. Choose $\www{b}\in S_{l_2}$ with 
$[\www{b}]_{[l_2]}=b$, and choose $\www{a_i}\in R^{l_1}$ with 
$[\www{a_i}]^{[l_1]}=a_i$. 
Then $\www{a_i}\cdot \www{b}\in S_{l_2-l_1-1}$ for each $i$.
Therefore $R^{l_1}\cdot \www{b}\subset S_{l_2-l_1-1}$, so 
$$R^{l_2-1}\cdot \www{b} = R^{l_2-l_1-1}\cdot R^{l_1}\cdot 
\www{b}\subset R^{l_2-l_1-1}\cdot S_{l_2-l_1-1}=\{0\}.$$
This implies $\www{b}\in S_{l_2-1}$, so $b=0$.
This proves \eqref{3.10}. \hfill$\Box$

\begin{example}\label{t3.4}
Consider $A=A^{(1)}\oplus A^{(2)}\oplus A^{(3)}$ with
\begin{eqnarray*}
A^{(1)}&\cong& \Q[x,y]/(xy,x^4,y^3),\\
A^{(2)}&\cong& \Q[z]/(z^2),\\
A^{(3)}&\cong& \Q.
\end{eqnarray*}
Let $\oooo{x},\oooo{x^2},\oooo{x^3},\oooo{y},\oooo{y^2},
\oooo{z}$ be the classes of the corresponding monomials
in $A^{(1)}$ respectively $A^{(2)}$. The figures \ref{fig3.1}
and \ref{fig3.2} show the radical filtration and the socle 
filtration. Here $n_1=3$, $n_2=1$, $n_3=0$, $n_{max}=3$. 
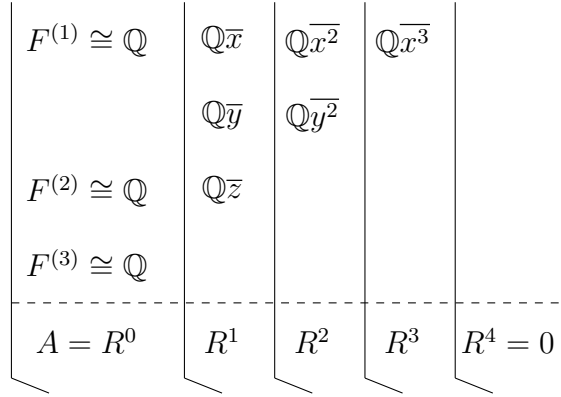
\begin{figure}
\begin{tikzpicture}
\draw [dashed] (-1,0.5)--(6.3,0.5);
\draw (-1,4.5)--(-1,-0.5);
\draw (-1,-0.5)--(-0.5,-0.7);
\node at (0,0) {$A=R^0$};
\node at (0,1) {$F^{(3)}\cong \Q$};
\node at (0,2) {$F^{(2)}\cong \Q$};
\node at (0,4) {$F^{(1)}\cong \Q$};
\draw (1.3,4.5)--(1.3,-0.5);
\draw (1.3,-0.5)--(1.8,-0.7);
\node at (1.8,0) {$R^1$};
\node at (1.8,2) {$\Q\oooo{z}$};
\node at (1.8,3) {$\Q\oooo{y}$};
\node at (1.8,4) {$\Q\oooo{x}$};
\draw (2.5,4.5)--(2.5,-0.5);
\draw (2.5,-0.5)--(3,-0.7);
\node at (3,0) {$R^2$};
\node at (3,3) {$\Q\oooo{y^2}$};
\node at (3,4) {$\Q\oooo{x^2}$};
\draw (3.7,4.5)--(3.7,-0.5);
\draw (3.7,-0.5)--(4.2,-0.7);
\node at (4.2,0) {$R^3$};
\node at (4.2,4) {$\Q\oooo{x^3}$};
\draw (4.9,4.5)--(4.9,-0.5);
\draw (4.9,-0.5)--(5.4,-0.7);
\node at (5.6,0) {$R^4=0$};
\end{tikzpicture}
\caption[Figure 3.1]{Radical filtration in Example \ref{t3.4}}
\label{fig3.1}
\end{figure}

\begin{figure}
\begin{tikzpicture}
\node at (-1.5,4) {$S_4$};
\node at (-1.5,3) {$S_3$};
\node at (-1.5,2) {$S_2$};
\node at (-1.5,1) {$S_1$};
\node at (-1.5,0) {$S_0$};
\draw (-1.8,3.5)--(-2,3);
\draw (-1.8,2.5)--(-2,2);
\draw (-1.8,1.5)--(-2,1);
\draw (-1.8,0.5)--(-2,0);
\draw (1.3,4.5)--(1.8,4.5);
\draw (2.5,4.5)--(3,4.5);
\draw (3.7,4.5)--(4.2,4.5);
\draw (4.9,4.5)--(5.4,4.5);
\draw (-1.8,0.5)--(1,0.5);
\draw (-1.8,1.5)--(1,1.5);
\draw (-1.8,2.5)--(1,2.5);
\draw (-1.8,3.5)--(1,3.5);
\draw (1,0.5)--(1.3,1.5);
\draw (1,1.5)--(1.3,2.5);
\draw (1,2.5)--(1.3,3.5);
\draw (1,3.5)--(1.3,4.5);
\draw (1.3,1.5)--(2.2,1.5);
\draw (1.3,2.5)--(2.2,2.5);
\draw (1.3,3.5)--(2.2,3.5);
\draw (2.2,1.5)--(2.5,2.5);
\draw (2.2,2.5)--(2.5,3.5);
\draw (2.2,3.5)--(2.5,4.5);
\draw (2.5,2.5)--(3.4,2.5);
\draw (2.5,3.5)--(3.4,3.5);
\draw (3.4,2.5)--(3.7,3.5);
\draw (3.4,3.5)--(3.7,4.5);
\draw (3.7,3.5)--(4.6,3.5);
\draw (4.6,3.5)--(4.9,4.5);
\draw [dashed] (-1,4.5)--(-1,-0.2);
\node at (0,1) {$F^{(3)}\cong \Q$};
\node at (0,2) {$F^{(2)}\cong \Q$};
\node at (0,4) {$F^{(1)}\cong \Q$};
\node at (1.8,2) {$\Q\oooo{z}$};
\node at (1.8,3) {$\Q\oooo{y}$};
\node at (1.8,4) {$\Q\oooo{x}$};
\node at (3,3) {$\Q\oooo{y^2}$};
\node at (3,4) {$\Q\oooo{x^2}$};
\node at (4.2,4) {$\Q\oooo{x^3}$};
\end{tikzpicture}
\caption[Figure 3.2]{Socle filtration in Example \ref{t3.4}}
\label{fig3.2}
\end{figure}
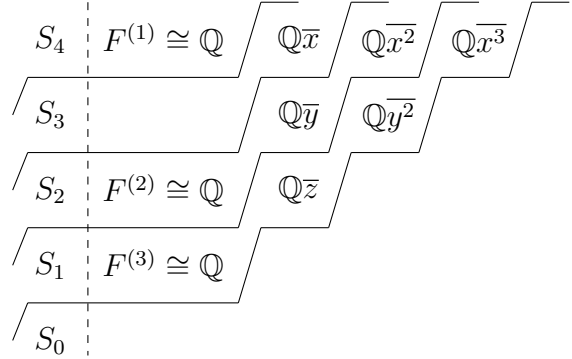
\end{example}

\section{Commutative semigroups}\label{s4}
\setcounter{equation}{0}
\setcounter{table}{0}

This section is a review of basic properties of
commutative semigroups. We follow essentially
\cite[ch. 1.2]{DTZ62}, but we condense the notions and results 
into two definitions and two theorems.  
They will be applied in section \ref{s5} and all later sections.

\begin{definition}\label{t4.1}
(a) A commutative semigroup is a set $S$ with a 
multiplication map $\cdot:S\times S\to S$ which is commutative
($ab=ba$) and associative ($a(bc)=(ab)c$). 
The semigroup is called $S$, so the multiplication map 
is suppressed.

(b) Let $S$ be a commutative semigroup. An element $a\in S$
is called {\it invertible} if an element $b\in S$ and an element
$c\in S$ with $ab=c$ and $ac=a$ exist.

(c) Let $S$ be a commutative semigroup. An element $c\in S$ is
called {\it idempotent} if $cc=c$. 
\end{definition}

\begin{theorem}\label{t4.2}
\cite[1.2.3, 1.2.10]{DTZ62}

Let $S$ be a commutative semigroup.

(a) Let $a\in S$ be invertible. Then there is a unique
element $c\in S$ with the properties $(ac=a,\exists\ b\in S
\textup{ with }ab=c)$. It is called $e_a$. It is idempotent.
There is a unique element $b\in S$ with the properties
$(ab=e_a,be_a=b)$. It is called $a^{-1}$. It is invertible, and 
$e_{a^{-1}}=e_a$. 

(b) An idempotent $c\in S$ is invertible with 
$e_c=c$ and $c^{-1}=c$.

(c) Let $c\in S$ be idempotent. The set 
\begin{eqnarray}\label{4.1}
G(c):=\{a\in S\,|\,a\textup{ invertible with }e_a=c\}
\end{eqnarray}
is a group. It is a maximal subgroup of $S$. 
Any maximal subgroup of $S$ is equal to $G(\www c)$
for some idempotent $\www c\in S$. 

(d) If $a,b\in S$ are invertible, then $ab$ is invertible
with $e_{ab}=e_ae_b$. 

(e) Suppose that $e_1,e_2\in S$ are both idempotent.
Then $e_1e_2$ is idempotent. The sets $G(e_1)e_2$, 
$G(e_2)e_1$ and $G(e_1)G(e_2)$ are subgroups of $G(e_1e_2)$.
The maps 
\begin{eqnarray}\label{4.2}
G(e_1)\to G(e_1)e_2,\ a\mapsto ae_2,\quad\textup{and}\quad
G(e_2)\to G(e_2)e_1,\ b\mapsto be_1,
\end{eqnarray}
are surjective group homomorphisms.

(f) The union $\bigcup_{c\textup{ idempotent}}G(c)\subset S$
of invertible elements in $S$ is a disjoint union. 
It is a subsemigroup of $S$. 
It is equal to $S$ if and only if any element of $S$ is invertible. 
\end{theorem}

{\bf Proof:}
(a) Suppose $b_1,c_1,b_2,c_2\in S$ with 
$ac_1=a,ab_1=c_1,ac_2=a,ab_2=c_2$. Then
\begin{eqnarray*}
c_1=ab_1=c_2ab_1=c_2c_1=b_2ac_1=b_2a=c_2,\\
c_1c_1=c_1ab_1=ab_1=c_1.
\end{eqnarray*}
This gives the uniqueness of $e_a$ and that it is idempotent.
Now suppose $b_3,b_4\in S$ with $ab_3=e_a, b_3e_a=b_3,
ab_4=e_a, b_4e_a=b_4$.
Then 
\begin{eqnarray*}
b_3=b_3e_a=b_3ab_4=e_ab_4=b_4.
\end{eqnarray*}
This gives the uniqueness of $a^{-1}$. Its existence follows
from 
\begin{eqnarray*}
a(b_1c_1)=(ab_1)c_1=c_1c_1=c_1\textup{ and }
(b_1c_1)c_1=b_1(c_1c_1)=b_1c_1.
\end{eqnarray*}
Its invertibility and 
$e_{a^{-1}}=e_a$ follow from $aa^{-1}=e_a,a^{-1}e_a=a^{-1}$.

(b) This follows from $cc=c$. 

(c) $c\in G(c)$ because of (b).
Any $a\in G(c)$ has in $G(c)$ the inverse $a^{-1}$.
For $a,b\in G(c)$ $(ab)a^{-1}b^{-1}=c$ and $(ab)c=ab$,
so also $ab\in G(c)$ and $(ab)^{-1}=a^{-1}b^{-1}$.
Therefore $G(c)$ is a group. If $G\subset S$ is a group,
then its unit element $e_G$ is idempotent and any element $a\in G$
is invertible with $e_a=e_G$, so $G\subset G(e_G)$.

(d) $(ab)a^{-1}b^{-1}=e_ae_b$, $(e_ae_b)(ab)=ab$. 

(e) This follows easily from the parts (b)--(d).

(f) If $a\in G(c_1)\cap G(c_2)$ with idempotents $c_1,c_2$,
then $c_1=e_a=c_2$. The rest follows from the parts (d) and (e). 
\hfill$\Box$

\begin{definition}\label{t4.3}
(a) Let $S$ be a commutative semigroup. 
A unit in it is an element $e$ with $ea=a$ for all $a\in S$. 
Obviously, it is unique if it exists.
A commutative semigroup with a unit is a {\it monoid}.

(b) Let $S$ be a commutative semigroup. 
Two elements $a_1$ and $a_2\in S$ are {\it $w$-equivalent} 
(notation: $a_1\sim_w a_2$),
if the following holds:
\begin{eqnarray}\label{4.3} a_1=a_2 \quad\textup{ or }\quad
\exists\ x_1,x_2\in S\textup{ with }a_1x_1=a_2,a_2x_2=a_1.
\end{eqnarray}
\end{definition}

The $w$-equivalence on a semigroup $S$ gives rise to a 
natural quotient semigroup $S/\sim_w=:W(S)$, as Theorem
\ref{t4.4} shows.

\begin{theorem}\label{t4.4}
\cite[1.2.12, 1.2.16, 1.2.17]{DTZ62}

Let $S$ be a commutative semigroup.

(a) $\sim_w$ is an equivalence relation and is compatible
with the multiplication in $S$, i.e.
\begin{eqnarray}\label{4.4}
a_1\sim_w a_2,b_1\sim_w b_2\Rightarrow a_1b_1\sim_w a_2b_2.
\end{eqnarray}
The equivalence class of $a$ is called $[a]_w$.
Therefore the quotient $S/\sim_w$ is a commutative semigroup. 
This semigroup is called $W(S)$. 

(b) If $a$ is invertible then $[a]_w=G(e_a)$. 

(c) Let $a\in S$. The following three properties are equivalent:
\begin{list}{}{}
\item[(i)] $a$ is invertible.
\item[(ii)] $[a]_w$ is invertible (in $W(S)$).
\item[(iii)] $[a]_w$ is idempotent (in $W(S)$).
\end{list}
The only subgroups of $W(S)$ are the sets $\{[c]_w\}$ with
$c\in S$ idempotent. 

(d) The subsemigroup $\bigcup_{c\textup{ idempotent}}G(c)$
(of invertible elements) of $S$ was considered in 
Theorem \ref{t4.2} (e). 
The subsemigroup $W(\bigcup_{c\textup{ idempotent}}G(c))$
of $W(S)$ is isomorphic to the subsemigroup
$\{c\in S\,|\, c\textup{ is idempotent}\}$ of $S$. 
\end{theorem}

{\bf Proof:}
(a) $\sim_w$ is reflexive and symmetric because of \eqref{4.3}.
In order to see that $\sim_w$
is transitive suppose $a_1\sim_w a_2$ and $a_2\sim_w a_3$.
If $a_1=a_2$ or $a_2=a_3$ then trivially $a_1\sim_w a_3$.
So suppose $a_1\neq a_2$ and $a_2\neq a_3$. 
Then $x_1,x_2,x_3,x_4$ with $a_1x_1=a_2,a_2x_2=a_1$ and
$a_2x_3=a_3,a_3x_4=a_2$ exist. Then
\begin{eqnarray*}
a_1x_1x_3=a_3,\ a_3x_4x_2=a_1,\quad\textup{so }a_1\sim_w a_3.
\end{eqnarray*}
Now suppose $a_1\sim_w a_2$ and $b_1\sim_w b_2$. 
If $a_1\neq a_2$ and $b_1\neq b_2$ then
$x_1,x_2,y_1,y_2\in S$ with $a_1x_1=a_2,a_2x_2=a_1$ and
$b_1y_1=b_2,b_2y_2=b_1$ exist. Then 
\begin{eqnarray*}
a_1b_1x_1y_1=a_2b_2,\ a_2b_2x_2y_2=a_1b_1,\quad\textup{so }
a_1b_1\sim_w a_2b_2.
\end{eqnarray*}
If $a_1=a_2$ or $b_1=b_2$, then $a_1b_1\sim_w a_2b_2$ follows
similarly.

(b) Let $a$ be invertible. 
First suppose $b\in[a]_w$ and $b\neq a$. 
Then $x_1,x_2\in S$ with $ax_1=b,bx_2=a$ exist. Then
\begin{eqnarray*}
e_ab=e_a ax_1=ax_1=b,\ b(x_2a^{-1})=aa^{-1}=e_a,\quad\textup{so }
b\in G(e_a).
\end{eqnarray*}
Now suppose $b\in G(e_a)$. Then $bb^{-1}=e_a,e_ab=b$ and 
\begin{eqnarray*}
b(b^{-1}a)=e_aa=a,\ a(a^{-1}b)=e_ab=b,\quad\textup{so }
b\in [a]_w.
\end{eqnarray*}

(c) (i)$\Rightarrow$(iii): Because of (b), $[a]_w=[e_a]_w$.
Because of (a), $[e_a]_w^2=[e_a^2]_w=[e_a]_w$. 

(iii)$\Rightarrow$(ii): An idempotent element in a semigroup
(here $W(S)$) is invertible, see Theorem \ref{t4.2} (b). 

(ii)$\Rightarrow$(i): Let $b,c\in S$ with 
$[a]_w[b]_w=[c]_w$ and $[c]_w[a]_w=[a]_w$. 
Then $[ab]_w=[c]_w$ and $[ca]_w=[a]_w$. 
If $ab\neq c$ and $ca\neq a$ then especially 
$x_1,y_1\in S$ with $(ab)x_1=c$ and $(ca)y_1=a$ exist, so
$a(bx_1y_1)=cy_1$ and $(cy_1)a=a$. This shows that then $a$
is invertible with $e_a=cy_1$. 
If $ab=c$ or $ca=a$, then $a$ invertible follows similarly
(one replaces $x_1$ or $y_1$ by an empty place).

By (ii)$\Rightarrow$(iii), a group $G([c]_w)$ with 
$[c]_w$ idempotent consists only of $[c]_w$. 
By Theorem \ref{t4.2} (c), these groups are the maximal
groups in $W(S)$, so they are the only groups in $W(S)$. 

(d) This follows from the parts (b) and (c) and Theorem
\ref{t4.2}. \hfill$\Box$

\section{Full lattices and orders in $\Q$-algebras,
some semigroups}\label{s5}
\setcounter{equation}{0}
\setcounter{table}{0}

Throughout the whole paper $A$ is as in Theorem \ref{t3.1}, 
so $A$ is a finite dimensional commutative $\Q$-algebra 
with unit element $1_A$.

\begin{definition}\label{t5.1}
An {\it order} $\Lambda$ in $A$ is a full lattice in $A$ with
\begin{eqnarray}\label{5.1}
1_A\in\Lambda\quad\textup{and}\quad \Lambda\cdot\Lambda\subset
\Lambda,
\end{eqnarray}
so it is a full lattice and a subring of $A$ with unit element.
\eqref{5.1} implies $\Lambda\cdot\Lambda=\Lambda$.
\end{definition}

This section provides basic properties of orders and full
lattices in $A$. It turns out that $\LL(A)$ becomes a
commutative semigroup, and there are quotient semigroups
$\EE(A)$, $W(\LL(A))$ and $W(\EE(A))$. 
The notions from section \ref{s4} will be applied and compared.

Section \ref{s6} will go deeper into the theory and present
one main result of the paper on the structure of the
quotient semigroup $\EE(A)$.

Part (c) of the following lemma is a version of Krull's
lemma, which we took from \cite[1.3.2]{DTZ62}.

\begin{lemma}\label{t5.2}
(a) Let $L$, $L_1$ and $L_2$ be full lattices in $A$. Then
\begin{eqnarray*}
L_1\cdot L_2&:=& \{\sum_{i\in I}a_ib_i\,|\, I\textup{ a finite 
index set}, a_i\in L_1,b_i\in L_2\}\textup{ and}\\
L_1:L_2 &:=&\{a\in A\,|\, a\cdot L_2\subset L_1\}
\end{eqnarray*}
are full lattices in $A$, and 
\begin{eqnarray*}
\OO(L)&:=& L:L
\end{eqnarray*}
is an order in $A$. It is called {\it the order of }$L$.

(b) Let $L_1,L_2,L_3$ and $L_4$ be full lattices in $A$. Then
\begin{eqnarray}\label{5.2}
L_1L_2= L_2L_1,&& (L_1L_2)L_3= L_1(L_2L_3),\\
(L_1:L_2)L_3\subset (L_1L_3):L_2,&& 
(L_1:L_2)(L_3:L_4)\subset (L_1L_3):(L_2L_4).\nonumber
\end{eqnarray}

(c) (Krull's lemma, \cite[1.3.2]{DTZ62})
Let $\Lambda$ and $L$ be full lattices in $A$ with
$\Lambda\cdot\Lambda\subset\Lambda$ and $\Lambda\cdot L=L$.
Then $1_A\in\Lambda$, so $\Lambda$ is an order.
\end{lemma}

{\bf Proof:} 
(a) As $L_1$ generates $A$ over $\Q$, it contains a unit
$a\in A^{unit}$ in $A$ 
(in fact, it contains many such elements).
As $L_2$ generates $A$ over $\Q$, also $a\cdot L_2$ and 
$L_1\cdot L_2$ generate $A$ over $\Q$. As $L_1\cdot L_2$ is
a finitely generated $\Z$-module, it is a full lattice in $A$.

In order to see that $L_1:L_2$ is a full lattice in $A$ one
applies Lemma \ref{t2.3} (a) with
\begin{eqnarray*}
(V_1,V_2,V_3,\beta,L_1,L_3)_{\textup{in Lemma \ref{t2.3}}}
\sim(A,A,A,\textup{multiplication},L_2,L_1)_{\textup{here}}.
\end{eqnarray*}
Therefore also $\OO(L)=L:L$ is a full lattice. Obviously, 
it contains $1_A$ and is multiplication invariant, so it is
an order.

(b) Trivial.

(c) Choose a $\Z$-basis $(b_1,...,b_n)$ of $L$.
Because of $L=\Lambda\cdot L$, there are $a_{ij}\in \Lambda$
for $i,j\in\{1,...,n\}$ with 
\begin{eqnarray*}
(b_1,...,b_n)=(b_1,...,b_n)\cdot(a_{ij}),\quad\textup{so}\quad
0=(b_1,...,b_n)\cdot (\delta_{ij}-a_{ij}).
\end{eqnarray*}
Multiplying from the right with the adjoint 
of the matrix $(\delta_{ij}-a_{ij})$,
we obtain $b_k\cdot \det(\delta_{ij}-a_{ij})=0$,
so $A\cdot \det(\delta_{ij}-a_{ij})=0$, so
$\det(\delta_{ij}-a_{ij})=0$. Because of 
$\Lambda\cdot\Lambda\subset\Lambda$, all products of the
$a_{ij}$ are in $\Lambda$, so also $1_A\in\Lambda$. 
\hfill$\Box$

\bigskip
Some basic observations on products of orders and the orders
of products $L_1 L_2$ and quotients $L_1:L_2$ are as 
follows.

\begin{lemma}\label{t5.3}
(a) If $\Lambda_1$ and $\Lambda_2\in\LL(A)$ are orders,
then the order $\Lambda_1\Lambda_2$ contains $\Lambda_1$
and $\Lambda_2$, and it is the smallest order with this property.
So, any order which contains $\Lambda_1$ and $\Lambda_2$,
contains $\Lambda_1\Lambda_2$. 

(b) Let $L_1,L_2\in\LL(A)$ with $\OO(L_1)=\Lambda_1$ and
$\OO(L_2)=\Lambda_2$. Then
\begin{eqnarray}\label{5.3}
\OO(L_1L_2)\supset\Lambda_1\Lambda_2,\quad 
\OO(L_1:L_2)\supset\Lambda_1\Lambda_2.
\end{eqnarray}
If $L_1$ is invertible then
\begin{eqnarray}\label{5.4}
L_1:L_2 = L_1\cdot (\Lambda_1:L_2).
\end{eqnarray}
If $L_2$ is invertible then
\begin{eqnarray}\label{5.5}
L_1:L_2 = (L_1:\Lambda_2)\cdot L_2^{-1}.
\end{eqnarray}
\end{lemma}

{\bf Proof:} (a) $\Lambda_1\Lambda_2\supset\Lambda_1$
because of $1_A\in\Lambda_2$.

(b) $\Lambda_1(L_1L_2)=(\Lambda_1L_1)L_2= L_1L_2$
shows $\Lambda_1\subset\OO(L_1L_2)$, and analogously
$\Lambda_2\subset\OO(L_1L_2)$. 
With part (a) this implies $\Lambda_1\Lambda_2\subset
\OO(L_1L_2)$. 

$(\Lambda_1(L_1:L_2))L_2=\Lambda_1((L_1:L_2)L_2)\subset
\Lambda_1L_1=L_1$, so $\Lambda_1(L_1:L_2)\subset (L_1:L_2)$,
so $\Lambda_1\subset\OO(L_1:L_2)$. 

$(\Lambda_2(L_1:L_2))L_2=(L_1:L_2)(\Lambda_2L_2)=
(L_1:L_2)L_2\subset L_1$, so 
$\Lambda_2(L_1:L_2)\subset (L_1:L_2)$,
so $\Lambda_2\subset\OO(L_1:L_2)$. 

Suppose that $L_1$ is invertible. Then
\begin{eqnarray*}
(L_1(\Lambda_1:L_2))L_2 \subset L_1\Lambda_1=L_1,
\quad\textup{so } 
L_1(\Lambda_1:L_2)\subset L_1:L_2,\\
(L_1^{-1} (L_1:L_2)) L_2 \subset L_1^{-1} L_1=\Lambda_1,
\quad\textup{so }L_1^{-1} (L_1:L_2)\subset \Lambda_1:L_2,\\
\textup{so } L_1:L_2\subset L_1 (\Lambda_1:L_2).
\end{eqnarray*}

Suppose that $L_2$ is invertible. Then
\begin{eqnarray*}
((L_1:\Lambda_2)L_2^{-1})L_2 \subset (L_1:\Lambda_2)\Lambda_2
\subset L_1,\quad\textup{so } 
(L_1:\Lambda_2)L_2^{-1}\subset L_1:L_2,\\
((L_1:L_2)L_2)\Lambda_2 \subset (L_1:L_2)L_2\subset L_1,
\quad\textup{so }(L_1:L_2)L_2\subset L_1:\Lambda_2,\\
\textup{so } L_1:L_2\subset (L_1:\Lambda_2)L_2^{-1}.
\hspace*{2cm}\Box
\end{eqnarray*}

\begin{definition}\label{t5.4}
Recall from the Notations \ref{t1.5} that $A^{unit}$ is the
set of units in $A$.

(a) An equivalence relation $\sim_\varepsilon$ on $\LL(A)$
is defined as follows,
\begin{eqnarray*}
L_1\sim_\varepsilon L_2 &:\iff& a\in A^{unit}
\textup{ with }a\cdot L_1=L_2\textup{ exists.}
\end{eqnarray*}
The equivalence class of a full lattice $L$ with respect to
$\sim_\varepsilon$ is called $\varepsilon$-class of $L$ and
is denoted $[L]_\varepsilon$. The set of $\varepsilon$-classes
of full lattices is denoted 
$\EE(A) \ (=\LL(A)/\sim_\varepsilon)$. 

(b) A full lattice $L$ in $A$ with $\OO(L)\supset\Lambda$
for some order $\Lambda$ is called a {\it $\Lambda$-ideal}.
It is called an {\it exact $\Lambda$-ideal} if $\OO(L)=\Lambda$.
\end{definition}

\begin{lemma}\label{t5.5}
The multiplication of full lattices gives the structure
of a commutative semigroup on $\LL(A)$. It induces the
structure of a commutative semigroup on $\EE(A)$.
If $L_1\sim_\varepsilon L_2$, then $\OO(L_1)=\OO(L_2)$, so the order 
$\OO([L_1]_\varepsilon):=\OO(L_1)$ is well defined. 
Also the division map $(L_1,L_2)\mapsto L_1:L_2$ on $\LL(A)$
induces a division map 
$([L_1]_\varepsilon,[L_2]_\varepsilon)\mapsto [L_1]_\varepsilon:[L_2]_\varepsilon:=[L_1:L_2]_\varepsilon$ 
on $\EE(A)$. 
\end{lemma}

{\bf Proof:} $\LL(A)$ is with the multiplication of full lattices a commutative semigroup by \eqref{5.2}.

Suppose $L_1\sim_\varepsilon L_3$ and $L_2\sim_\varepsilon L_4$.
Then $a_1,a_2\in A^{unit}$ with $a_1L_1=L_3$ and $a_2L_2=L_4$
exist. Then $a_1a_2 L_1 L_2=L_3 L_4$
and $(a_1a_2^{-1})\cdot (L_1:L_2)=L_3:L_4$, so 
$L_1 L_2\sim_\varepsilon L_3 L_4$ and 
$L_1:L_2\sim_\varepsilon L_3:L_4$. \hfill$\Box$

\bigskip
Lemma \ref{t5.5} sets the stage for applying and comparing
the notions from section \ref{s4} to the semigroups
$\LL(A)$ and $\EE(A)$. The following theorem discusses the
idempotents and the invertible elements in both semigroups.
Remarkably, the division maps in $\LL(A)$ and $\EE(A)$ play
a much less prominent role than the multiplications.

The major part of Theorem \ref{t5.6} is essentially contained in 
\cite[section 1]{DTZ62}. Though the assumptions there mean in
our situation that $A$ is an algebraic number field.
Therefore we reprove Theorem \ref{t5.6}. The same holds
for Theorem \ref{t5.7}.
The equivalence (i)$\iff$(v) in Theorem \ref{t5.6} (c) 
is not in \cite{DTZ62}, but in \cite[26.4]{Fa65}. 
Though the proof below of (i)$\iff$(v) is not given
in \cite{Fa65}.

\begin{theorem}\label{t5.6}\cite[1.3.3, 1.3.6, 1.3.7]{DTZ62}
\cite[26.4]{Fa65}

(a) $\Lambda\in\LL(A)$ is an idempotent in the semigroup
$\LL(A)$ if and only if it is an order.

(b) $[L]_\varepsilon$ is an idempotent in the semigroup $\EE(A)$
if and only if the class $[L]_\varepsilon$ contains an order.

(c) Let $L\in\LL(A)$. The following five properties 
are equivalent:
\begin{list}{}{}
\item[(i)] 
$L$ is invertible in the semigroup $\LL(A)$.
\item[(ii)] 
$L_2\in\LL(A)$ with $L\cdot L_2=\OO(L)$ exists.
\item[(iii)] 
$L\cdot (\OO(L):L)=\OO(L)$. 
\item[(iv)]
$[L]_\varepsilon$ is invertible in the semigroup $\EE(A)$. 
\item[(v)]
$\OO(\OO(L):L)=\OO(L)$. 
\end{list}
If this holds then $e_L=\OO(L)$, $L^{-1}=\OO(L):L$,
$L^{-1}$ is invertible with $\OO(L^{-1})=\OO(L)$,
and $e_{[L]_\varepsilon}=[\OO(L)]_\varepsilon$, 
$[L]_{\varepsilon}^{-1}=[L^{-1}]_\varepsilon$.
\end{theorem}

{\bf Proof:}
(a) $\Leftarrow$: Let $\Lambda$ be an order. Because
of $1\in\Lambda$, we have $\Lambda\Lambda=\Lambda$ (and not just
$\Lambda\Lambda\subset\Lambda$), so $\Lambda$
is an idempotent. 

$\Rightarrow$: Let $\Lambda$ be an idempotent, so
$\Lambda\Lambda=\Lambda$. 
Lemma \ref{t5.2} (c) (Krull's lemma) for $\Lambda=L$ shows
$1\in\Lambda$. Therefore $\Lambda$ is an order.

(b) $\Leftarrow$: Let $\Lambda\in[L]_\varepsilon$ be an order.
Then
$$[L]_\varepsilon [L]_\varepsilon
=[\Lambda]_\varepsilon [\Lambda]_\varepsilon
=[\Lambda\Lambda]_\varepsilon
=[\Lambda]_\varepsilon 
=[L]_\varepsilon.$$

$\Rightarrow$: Suppose 
$[L]_\varepsilon[L]_\varepsilon=[L]_\varepsilon$.
Then $LL\sim_\varepsilon L$, so $a\in A^{unit}$ with 
$aLL=L$ exists. Then $(aL)^2=(aL)$, and $aL\in [L]_\varepsilon$ 
is an idempotent in $\LL(A)$. 
By part (a), $aL$ is an order. 

(c) (i)$\Rightarrow$(ii): Let $L$ be invertible in the
semigroup $\LL(A)$. Then $L_2$ and $\Lambda\in \LL(A)$
with $\Lambda$ an order and $LL_2=\Lambda$ and 
$\Lambda L=L$ exist. The last condition says 
$\Lambda\subset\OO(L)$. 
Now $LL_2=\OO(L)LL_2=\OO(L)\Lambda=\OO(L)$ 
($=$ and not just $\subset$ because of $1\in\Lambda$).

(ii)$\Rightarrow$(iii): By definition of $\OO(L):L$,
$L_2\subset\OO(L):L$ and $L\cdot(\OO(L):L)\subset\OO(L)$.
With $L\cdot L_2=\OO(L)$ this shows $L\cdot (\OO(L):L)=\OO(L)$. 

(iii)$\Rightarrow$(iv): 
$[L]_\varepsilon[\OO(L):L]_\varepsilon=[L(\OO(L):L)]_\varepsilon
=[\OO(L)]_\varepsilon$
and $[\OO(L)]_\varepsilon[L]_\varepsilon=[\OO(L)L]_\varepsilon
=[L]_\varepsilon$. 

(iv)$\Rightarrow$(i): 
Let $[L]_\varepsilon$ be invertible in the semigroup $\EE(A)$
with $e_{[L]_\varepsilon}=[\Lambda]_\varepsilon$ with $\Lambda$
an order. Then $[L]_\varepsilon=[L]_\varepsilon
[\Lambda]_\varepsilon=[L\Lambda]_\varepsilon$, and 
$L_1\in\LL(A)$ with 
$[L]_\varepsilon [L_1]_\varepsilon=[\Lambda]_\varepsilon$ exists. 
Thus $a\in A^{unit}$ with $aLL_1=\Lambda$ exists, so 
$\Lambda\supset\OO(L)$. Now $[L]_\varepsilon=
[L\Lambda]_\varepsilon$ shows $\Lambda=\OO(L)$, so $L$ is invertible in $\LL(A)$.

Now suppose that (i)--(iv) hold. 
(iii) and $\OO(L)L=L$ show that $L$ is invertible with
$e_L=\OO(L)$.
\eqref{5.3} shows $\OO(L)(\OO(L):L)=\OO(L):L$. 
With (iii) this also shows 
$L^{-1}=\OO(L):L$ and $\OO(L^{-1})=\OO(L)$.
The properties of the classes in $\EE(A)$ are clear.

(i)$\Rightarrow$(v): 
(i) implies $\OO(L):L=L^{-1}$ and $\OO(L^{-1})=\OO(L)$. 

(v)$\Rightarrow$(i): 
We will apply Lemma \ref{t2.3} with
\begin{eqnarray*}
&&(V_1,V_2,V_3,\beta,L_1,L_3)_{\textup{Lemma \ref{t2.3}}}\\
&= &(A,A,A,\textup{multiplication},L(\OO(L):L),
\OO(L))_{\textup{here}}.
\end{eqnarray*}
We will write $\beta^*:A\times A^*\to A^*$ also as 
multiplication.
We will show
\begin{eqnarray}
L_1^2&\subset& L_1,\label{5.6}\\
L_1\cdot L_3^{*\Z}&=& L_3^{*\Z}.\label{5.7}
\end{eqnarray}
A variant of Krull's lemma Lemma \ref{t5.2} (c) with the lattice
$L_3^{*\Z}\in\LL(A^*)$ instead of a full lattice in $A$ works and
shows $1_A\in L_1=L(\OO(L):L)$, so $L_1=\OO(L)$, so $L$ 
is invertible.

\eqref{5.6} follows from
\begin{eqnarray*}
L_1^2= (L(\OO(L):L))^2\subset \OO(L)\cdot L(\OO(L):L)
= L(\OO(L):L)=L_1.
\end{eqnarray*}
\eqref{5.7} is by Lemma \ref{t2.3} (b) equivalent to
\begin{eqnarray}\label{5.8}
L_3:L_1=L_3.
\end{eqnarray}
This is proved as follows. Consider $a\in L_3:L_1
=\OO(L):(L(\OO(L):L))$. 
\begin{eqnarray*}
&&L(\OO(L):L)a\subset \OO(L),\\
\textup{so}&& (\OO(L):L)a\textup{ maps }L\textup{ to }\OO(L),\\
\textup{so}&& (\OO(L):L)a\subset \OO(L):L,\\
\textup{so}&& a\in \OO(\OO(L):L) =\OO(L)
\textup{ (here (v) is used)},\\
\textup{so}&& \OO(L):(L(\OO(L):L))=\OO(L).
\end{eqnarray*}
The implication (v)$\Rightarrow$(i) is proved. \hfill$\Box$

\bigskip
Now we come to the weak equivalence classes and the 
semigroups $W(\LL(A))$ and $W(\EE(A))$. Part (c) of Theorem
\ref{t5.7} says that they coincide.

\begin{theorem}\label{t5.7}
\cite[1.3.9, 1.3.11]{DTZ62}

(a) Let $L_1,L_2\in\LL(A)$. The following four properties
are equivalent.
\begin{list}{}{}
\item[(i)]
$L_1\sim_w L_2$.
\item[(ii)]
$1\in (L_1:L_2)(L_2:L_1)$.
\item[(iii)]
$\OO(L_1)=\OO(L_2)$ and $L_3\in G(\OO(L_1))$ with
$L_1L_3=L_2$ exists.
\item[(iv)]
$[L_1]_\varepsilon\sim_w[L_2]_\varepsilon$.
\end{list}

(b) Let $L_1,L_2\in\LL(A)$ with $L_1\sim_w L_2$. We have
\begin{eqnarray*}
L_1:L_2\in G(\OO(L_1)),\quad L_2:L_1\in G(\OO(L_1)),
\quad (L_1:L_2)^{-1}=L_2:L_1,\\
L_2=(L_2:L_1)L_1,\quad L_1=(L_1:L_2)L_2,\\
\OO(L_1)=\OO(L_2)=\OO(L_1:L_2)=\OO(L_2:L_1)
=(L_1:L_2)(L_2:L_1).
\end{eqnarray*}

(c) $W(\LL(A))=W(\EE(A))$, and this semigroup
inherits a division map from the division map on $\LL(A)$. 
\end{theorem}

{\bf Proof:}
(a) (i)$\Rightarrow$(ii): 
Suppose $L_1\sim_w L_2$. Then $L_3,L_4\in\LL(A)$ with
$L_1L_3=L_2$ and $L_2L_4=L_1$ exist. Especially
\begin{eqnarray*}
L_1(L_2:L_1)=L_2\quad\textup{and}\quad L_2(L_1:L_2)=L_1
\end{eqnarray*}
(and not just $L_1(L_2:L_1)\subset L_2$ and 
$L_2(L_1:L_2)\subset L_1$). 
Lemma \ref{t5.2} (c) (Krull's lemma) can be applied 
and yields $1\in (L_1:L_2)(L_2:L_1)$, because of the following
two calculations,
\begin{eqnarray*}
&&L_2(L_1:L_2)(L_2:L_1)=L_1(L_2:L_1)=L_2,\\
&&\bigl((L_1:L_2)(L_2:L_1)\bigr)^2 =
(L_1:L_2)(L_2:L_1)(L_1:L_2)(L_2:L_1)\\
&\subset& (L_1:L_2)(L_2:L_1)(L_1:L_1)
=(L_1:L_2)(L_2:L_1)\OO(L_1)\\
&\stackrel{\eqref{5.3}}{=}& (L_1:L_2)(L_2:L_1).
\end{eqnarray*}

(ii)$\Rightarrow$(iii): Define $\Lambda:=(L_1:L_2)(L_2:L_1)$.
The last calculation $\Lambda\Lambda\subset \Lambda$
and the assumption $1\in\Lambda$ in (ii) 
show that $\Lambda$ is an order. 
It contains $\OO(L_1)$ and $\OO(L_2)$ because of \eqref{5.3}. 
It is contained in $L_1:L_1=\OO(L_1)$
and in $L_2:L_2=\OO(L_2)$ by definition of $L_1:L_2$
and $L_2:L_1$. Therefore $\OO(L_1)=\Lambda=\OO(L_2)$.
This is also equal to $\OO(L_1:L_2)$ and $\OO(L_2:L_1)$,
because \eqref{5.3} implies both of the following 
inclusions,
\begin{eqnarray*}
\OO(L_1)\subset \OO(L_1:L_2)\subset \OO(\Lambda)=\Lambda.
\end{eqnarray*}
With $\Lambda=(L_1:L_2)(L_2:L_1)$ we obtain 
$L_1:L_2\in G(\Lambda)$, $L_2:L_1\in G(\Lambda)$ and
$(L_1:L_2)^{-1}=L_2:L_1$. Now $L_3:=L_2:L_1\in G(\Lambda)$
satisfies 
\begin{eqnarray*}
L_2=L_2\Lambda\subset L_1(L_2:L_1)\subset L_2,
\quad\textup{so }L_2=L_1(L_2:L_1)=L_1L_3.
\end{eqnarray*}
 
(iii)$\Rightarrow$(i): 
$L_1L_3=L_2$ and $L_2L_3^{-1}=L_1L_3L_3^{-1}=L_1\OO(L_1)=L_1$
show $L_1\sim_w L_2$.

(i)$\Rightarrow$(iv): Trivial.

(iv)$\Rightarrow$(i): Suppose 
$[L_1]_\varepsilon[L_3]_\varepsilon=[L_2]_\varepsilon$ and
$[L_2]_\varepsilon[L_4]_\varepsilon=[L_1]_\varepsilon$.
Then $a_1,a_2\in A^{unit}$ with 
$a_1L_1L_3=L_2$ and $a_2L_2L_4=L_1$ exist. Thus
$L_1\sim_w L_2$.

(b) This was shown in the proof of (ii)$\Rightarrow$(iii)
in part (a).

(c) Suppose $L_1\sim_\varepsilon L_2$. Then $\OO(L_1)=\OO(L_2)$
by Lemma \ref{t5.5}, and $a\in A^{unit}$ with $aL_1=L_2$ exists. Then $L_1a\OO(L_1)=L_2$ and $L_2a^{-1}\OO(L_1)=L_1$, so 
$L_1\sim_w L_2$. Therefore the sets $W(\LL(A))$ and
$W(\EE(A))$ coincide. Also the multiplications coincide,
because all multiplications are induced from the multiplication
on $\LL(A)$. Finally, we want to show 
\begin{eqnarray}\label{5.9}
L_3\sim_w L_5,\ L_4\sim_w L_6\Rightarrow 
L_3:L_4\sim_w L_5:L_6.
\end{eqnarray}
We calculate 
\begin{eqnarray*}
L_3:L_4&\supset& (L_3:L_5)(L_5:L_6)(L_6:L_4)\\
&\supset& (L_3:L_5)(L_5:L_3)(L_3:L_4)(L_4:L_6)(L_6:L_4)\\
&\stackrel{\textup{(b)}}{=}&
\OO(L_3)(L_3:L_4)\OO(L_4)
\stackrel{\textup{\eqref{5.3}}}{=}(L_3:L_4),\\
\textup{so }L_3:L_4&=&(L_5:L_6)\bigl((L_3:L_5)(L_6:L_4)\bigr),
\end{eqnarray*}
and analogously 
\begin{eqnarray*}
L_5:L_6&=&(L_3:L_4)\bigl((L_5:L_3)(L_4:L_6)\bigr).
\end{eqnarray*}
Therefore $L_3:L_4\sim_w L_5:L_6$.\hfill$\Box$

\bigskip
By definition, $L_1\sim_w L_2$ if and only
if $L_3$ and $L_4\in\LL(A)$ exist with 
$L_1L_3=L_2$ and $L_2L_4=L_1$. 
Theorem \ref{t5.7} (a) (i)$\iff$(iii) and (b) gives a remarkable
strengthening of this. It says that one can choose
$L_3$ and $L_4$ in $G(\OO(L_1))$, namely $L_3=L_2:L_1$
and $L_4=L_1:L_2=L_3^{-1}$. 

In fact, $L_3$ and $L_4$ with this additional condition
$L_3,L_4\in G(\OO(L_1))$ are unique, see part (a)
of the next Theorem \ref{t5.8}. 
The proof of this is easy. But it has the important
consequence, which is formulated in part (b) of Theorem 
\ref{t5.8} that for any $L\in\LL(A)$ there are bijections
$[\OO(L)]_w\to [L]_w$ and 
$[[\OO(L)]_\varepsilon]_w\to [[L]_\varepsilon]_w$.
Theorem \ref{t5.8} is not stated in \cite{DTZ62}.

Consider an order $\Lambda$ and the set
\begin{eqnarray}\label{5.10}
\{[L]_\varepsilon\,|\, L\in\LL(A),\OO(L)=\Lambda\}.
\end{eqnarray}
Theorem \ref{t6.5} will say that this set is finite.
Then \eqref{5.13} in Theorem \ref{t5.8} says that this set 
decomposes into finitely many $w$-classes which have all 
the same size, one of them being $G([\Lambda]_\varepsilon)$. 
The Theorems \ref{t6.5} and \ref{t5.8} (b)
together structure $\EE(A)$ in a good way.

\begin{theorem}\label{t5.8}
(a) Let $L_1,L_2\in\LL(A)$ with $L_1\sim_w L_2$.
Then $\OO(L_1)=\OO(L_2)$. The only full lattice $L_3$ with
$L_1L_3=L_2$ and $L_3\in G(\OO(L_1))$ is $L_3=L_2:L_1$. 

(b) Let $\Lambda$ be an order and $L_1\in \LL(A)$
with $\OO(L_1)=\Lambda$. Then 
\begin{eqnarray}
G(\Lambda)=[\Lambda]_w\quad\textup{and}\quad 
G([\Lambda]_\varepsilon)=[[\Lambda]_\varepsilon]_w.\label{5.11}
\end{eqnarray}
The maps 
\begin{eqnarray}
G(\Lambda)\to [L_1]_w,&& L_3\mapsto L_3L_1,\label{5.12}\\
G([\Lambda]_\varepsilon)\to [[L_1]_\varepsilon]_w,
&& [L_3]_\varepsilon\mapsto [L_3]_\varepsilon[L_1]_\varepsilon,
\label{5.13}
\end{eqnarray}
are well defined and bijections. 
\end{theorem}

{\bf Proof:} 
(a) Theorem \ref{t5.7} (a) (i)$\Rightarrow$(iii) gives
$\OO(L_1)=\OO(L_2)$. Suppose $L_3\in G(\OO(L_1))$ with
$L_1L_3=L_2$. Then $L_3\subset L_2:L_1$.
Also $L_1=L_2L_3^{-1}$ which implies $L_3^{-1}\subset L_1:L_2$.
This last inclusion implies
$$L_3\supset (L_1:L_2)^{-1} 
\stackrel{\textup{Theorem \ref{t5.7} (b)}}{=}L_2:L_1.$$
Therefore $L_3=L_2:L_1$. 

(b) \eqref{5.11} follows from Theorem \ref{t4.4} (b).

The map in \eqref{5.12} has indeed image in $[L_1]_w$ because
for $L_3\in G(\Lambda)$
\begin{eqnarray*}
L_1\cdot L_3=L_3L_1,\quad 
(L_3L_1)\cdot L_3^{-1}=L_1,
\quad\textup{so }L_3L_1\sim_w L_1.
\end{eqnarray*}
It is surjective because of (i)$\Rightarrow$(iii) in Theorem
\ref{t5.7} (a). It is injective because of part (a). 

The map in \eqref{5.13} is bijective because the map in 
\eqref{5.12} is bijective and because it respects 
$\varepsilon$-classes.
\hfill$\Box$

\section{An extension of the Jordan-Zassenhaus theorem}\label{s6}
\setcounter{equation}{0}
\setcounter{table}{0}

Throughout the whole paper $A$ is as in Theorem \ref{t3.1}, 
so $A$ a finite dimensional commutative $\Q$-algebra 
with unit element $1_A$.
It is {\it separable} if its radical $R$ is $0$.
Then $A$ is a direct sum of algebraic number fields.
The following theorem
collects some classical results from algebraic number theory.

\begin{theorem}\label{t6.1}
(E.g. \cite{Ne99} or \cite{BSh73})

Let $A$ be an algebraic number field.

(a) It has a maximal order $\Lambda_{max}(A)$, which contains 
all other orders, namely the set of algebraic integers in $A$.

(b) Each full lattice $L$ with $\OO(L)=\Lambda_{max}(A)$
is invertible.

(c) The group 
$$G([\Lambda_{max}]_\varepsilon)=
\{[L]_\varepsilon\,|\, L\textup{ is a full lattice with }
\OO(L)=\Lambda_{max}(A)\}$$
is finite. It is the class group of $A$.
\end{theorem}

\begin{corollary}\label{t6.2}
Let $A$ be separable, so $A=\bigoplus_{j=1}^k A^{(j)}$ with 
$A^{(1)},...,A^{(k)}$ algebraic number fields.
Then Theorem \ref{t6.1} holds also in this situation. 
More precisely, the following holds.

(a) $A$ has a maximal order $\Lambda_{max}(A)$, which contains
all other orders. It is
$\Lambda_{max}(A)=\bigoplus_{j=1}^k\Lambda_{max}(A^{(j)})$.

(b) Each full lattice $L$ with $\OO(L)=\Lambda_{max}(A)$ is a 
direct sum $L=\bigoplus_{j=1}^kL^{(j)}$ where $L^{(j)}$ is a 
full lattice in $A^{(j)}$ with 
$\OO(L^{(j)})=\Lambda_{max}(A^{(j)})$.
$L$ is invertible with inverse
$L^{-1}=\bigoplus_{j=1}^k (L^{(j)})^{-1}$.

(c) The group $G([\Lambda_{max}]_\varepsilon)$ is finite and
isomorphic to the product 
$\prod_{j=1}^k G([\Lambda_{max}(A^{(j)})]_\varepsilon)$ of the
class groups of the algebraic numbers fields
$A^{(1)},...,A^{(k)}$.
\end{corollary}

{\bf Proof:}
(a) If $\Lambda$ is an order in $A$, then 
$1_{A^{(j)}}\cdot \Lambda$ is an order in $A^{(j)}$.
Then $\Lambda\subset\bigoplus_{j=1}^k1_{A^{(j)}}\cdot\Lambda
\subset\bigoplus_{j=1}^k \Lambda_{max}(A^{(j)})
=\Lambda_{max}.$

(b) $1_{A^{(j)}}\in \Lambda_{max}(A^{(j)})
\subset\Lambda_{max}(A)$. Therefore $L$ contains 
$L^{(j)}:=1_{A^{(j)}}\cdot L$ and is the direct sum
$L=\bigoplus_{j=1}^kL^{(j)}$. Also
$$\OO(L^{(j)})=1_{A^{(j)}}\cdot\OO(L)
=1_{A^{(j)}}\cdot \Lambda_{max}(A)
\stackrel{\textup{(a)}}{=}\Lambda_{max}(A^{(j)}).$$
By part (b) of Theorem \ref{t6.1} $L^{(j)}$ is invertible in
$A^{(j)}$. The rest is clear.

(c) This follows from Theorem \ref{t6.1} (c) and from 
part (b).\hfill$\Box$

\bigskip
The finiteness of the class group in Theorem \ref{t6.1} (c) 
is a very special case of the Jordan-Zassenhaus 
theorem \cite{Za38} (see also \cite[Theorem (79.1)]{CR62}
or \cite[Theorem (26.4)]{Re03}). This theorem holds also
for noncommutative separable algebras. In our situation
it implies the following finiteness result.

\begin{theorem}\label{t6.3}
(Special case of the Jordan-Zassenhaus theorem)

Let $A$ be separable.
For any order $\Lambda$ in $A$ the set
$$\{[L]_\varepsilon\,|\, L\in\LL(A),\OO(L)\supset \Lambda\}$$
of $\varepsilon$-classes of $\Lambda$-ideals is finite.
\end{theorem}

Theorem \ref{t6.3} does not hold for $A$ not separable,
as the following construction of Faddeev shows.

\begin{theorem}\label{t6.4}
\cite[Proposition 25.1]{Fa65}

Suppose that $A$ is not separable, so its radical $R$ is not 
$0$. Let $\Lambda$ be an order in $A$.
Then the set $\{\Gamma\in\LL(A)\,|\, 
\Gamma\textup{ an order}, \Gamma\supset\Lambda\}$ is infinite.
And therefore the set 
$\{[L]_\varepsilon\,|\, L\in\LL(A),\OO(L)\supset\Lambda\}$
is infinite.
\end{theorem}

{\bf Proof:} For example, the following is an infinite 
sequence of orders $\Lambda_m$ 
with $\Lambda\subsetneqq \Lambda_1\subsetneqq\Lambda_2
\subsetneqq ...\subsetneqq\Lambda_m\subsetneqq\Lambda_{m+1}
\subsetneqq...$,
\begin{eqnarray*}
\Lambda_m:=\Lambda+\sum_{l=1}^{n_{max}}\frac{1}{2^{lm}}
\bigl(\Lambda\cap R\bigr)^l\quad\textup{for }m\in\N
\end{eqnarray*}
where $n_{max}\in\N$ is unique with
$R^{n_{max}}\supsetneqq R^{n_{max}+1}=\{0\}$.
The second statement follows from the first statement and from
the following basic observation:
\begin{eqnarray}\label{6.1}
\Lambda_1,\Lambda_2\in\LL(A)\textup{ orders 
with }[\Lambda_1]_\varepsilon=[\Lambda_2]_\varepsilon
\quad\Rightarrow\quad \Lambda_1=\Lambda_2,
\end{eqnarray}
which follows from $\Lambda_1=\OO(\Lambda_1)=
\OO(\Lambda_2)=\Lambda_2$. 
\hfill$\Box$

\bigskip
But if one restricts to {\it exact} $\Lambda$-ideals,
Theorem \ref{t6.3} generalizes to not separable algebras.
This extension of the Jordan-Zassenhaus theorem in Theorem
\ref{t6.5} is one of the five main results of this
paper. In the case of an algebra $A$ with radical $R$ with
$R^2=0$ it was proved by Faddeev \cite{Fa67}.

\begin{theorem}\label{t6.5}
Let $A$ be as in Theorem \ref{t3.1}, so $A$ is a finite 
dimensional commutative $\Q$-algebra with unit element.
For any order $\Lambda$ in $A$, the set
$$\{[L]_\varepsilon\,|\, L\in\LL(A),\OO(L)= \Lambda\}$$
of $\varepsilon$-classes of exact $\Lambda$-ideals is finite.
\end{theorem}

The proof is given after the following remarks.

\begin{remarks}\label{t6.6}
(i) In the case of a separable algebra $A$, Theorem \ref{t6.5}
is equivalent to Theorem \ref{t6.3} because between a given
order $\Lambda$ and the maximal order $\Lambda_{max}(A)$
there are only finitely many orders.

(ii) Theorem \ref{t6.5} does not hold for a noncommutative
$\Q$-algebra with radical $R\neq 0$, not even if $R^2=0$.
Faddeev gives an example \cite[$7^\circ$]{Fa67}.

(iii) Our proof below of Theorem \ref{t6.5} and Faddeev's proof
for the case with $R^2=0$ in \cite{Fa67} have both three steps.
In \cite{Fa67} they are given in the sections
$4^\circ$, $5^\circ$ and $6^\circ$. 

The steps 1 coincide essentially.
Though our choice of $L$ with $1_A\in L$ is not made in 
\cite[$4^\circ$]{Fa67}.

Our step 2 follows the argument of Faddeev's step 2 closely.

But our step 3 proceeds differently. It uses the socle 
filtration in Lemma \ref{t3.3} together with the action of
$R^{[1]}$ on it (also in Lemma \ref{t3.3}) and Lemma \ref{t2.3}.
It is less technical and generalizes more easily to algebras
$A$ with $R^2\neq 0$ than Faddeev's step 3.
\end{remarks}

{\bf Proof of Theorem \ref{t6.5}:} Fix an order $\Lambda$ in $A$. 

{\bf Step 1:} Recall from Theorem \ref{t3.1} the decompositions
$A=\bigoplus_{j=1}^kA^{(j)}$ and $A^{(j)}=F^{(j)}\oplus N^{(j)}$,
the separable subalgebra $F:=\bigoplus_{j=1}^k F^{(j)}$ of $A$,
the radical $R=\bigoplus_{j=1}^k N^{(j)}$ of $A$, and the 
decomposition $A=F\oplus R$. 
Therefore there is a natural isomorphism
\begin{eqnarray*}
F\cong R^{[0]}=R^0/R^1=A/R.
\end{eqnarray*}
Here and in the following we use the notations from Lemma
\ref{t3.3} for the socle filtration $S_\bullet$ on $A$, 
the radical filtration $R^\bullet$ on $A$ 
and the corresponding quotients $S_{[l]}$ and $R^{[l]}$. 

We extend the notations as follows.
For each full lattice $L$ in $A$, its image under the
projection to $R^{[0]}$ is called $L^{[0]}$.
It is a full lattice in $R^{[0]}$ and is identified with a 
full lattice in $F$. For $m\in\{1,2,...,n_{max}\}$ the subquotient
$$L_{[m]}:= (L\cap S_m+S_{m-1})/S_{m-1}\subset S_m/S_{m-1}=S_{[m]}$$
is a full lattice in $S_{[m]}$. This notation $L_{[m]}$ is compatible
with the one in Lemma \ref{t2.4}, with $V=A$ and $V_\bullet =S_\bullet$.

$\Lambda^{[0]}$ is an order in $F$. 
If $L$ is an exact $\Lambda$-ideal then the full lattice 
$L^{[0]}$ in $F$ is a $\Lambda^{[0]}$-ideal (but not 
necessarily an exact $\Lambda^{[0]}$-ideal).
An $\varepsilon$-class of full lattices in $A$ projects
by $L\mapsto L^{[0]}$ to an $\varepsilon$-class of full lattices
in $F$.

By the Jordan-Zassenhaus theorem \ref{t6.3} the set of 
$\varepsilon$-classes of $\Lambda^{[0]}$-ideals in $F$ is finite,
so it can be indexed by a finite set $I$.
We choose for each $i\in I$ a representative $K^i\in\LL(F)$.
We can choose it and will choose it such that $1_A\in K^i$.

Each $\varepsilon$-class of exact $\Lambda$-ideals contains
full lattices $L$ with $L^{[0]}=K^i$ for some $i\in I$.
If $1_A\notin L$ for such an $L$, 
then $1_A+r\in L$ for some $r\in R$,
because $L^{[0]}=K^i$ and $1_A\in K^i$.
Then we replace $L$ by $\www{L}:=(1_A+r)^{-1}L$. 
It satisfies $[\www{L}]_\varepsilon=[L]_\varepsilon$,
${\www{L}}^{[0]}=L^{[0]}=K^i$ and $1_A\in \www{L}$.

Finally, for an exact $\Lambda$-ideal $L$ the condition
$1_A\in L$ is equivalent to the condition $L\supset \Lambda$.

Therefore it is sufficient to show for each $i\in I$ that the
set 
\begin{eqnarray}\label{6.2}
\{[L]_\varepsilon\,|\, L\in\LL(A)\textup{ with }\OO(L)=\Lambda,
L^{[0]}=K^i,L\supset\Lambda\}
\end{eqnarray}
is finite. In fact, we will show for each $i\in I$ that the set 
\begin{eqnarray}\label{6.3}
\{L\,|\, L\in\LL(A)\textup{ with }\OO(L)=\Lambda,
L^{[0]}=K^i,L\supset\Lambda\}
\end{eqnarray}
is finite.

{\bf Step 2:}
Recall from Lemma \ref{t3.3} that $S_{[1]}=S_1=\Ann(R)\subset A$
is the socle of $A$, that $\Lambda_{[1]}=\Lambda\cap S_1$
and that $L_{[1]}=L\cap S_1$ for any full lattice $L$ in $A$.
In step 2 we will show for each $i\in I$ that the set
\begin{eqnarray}\label{6.4}
\{L_{[1]}\,|\, L\in\LL(A)\textup{ with }\OO(L)=\Lambda,
L^{[0]}=K^i,L\supset\Lambda\}
\end{eqnarray}
is finite. For $L$ with $\OO(L)=\Lambda$ and $L^{[0]}=K^i$
the condition $\Lambda=L:L$ and the multiplication map
$R^{[0]}\times S_1\to S_1$ in Lemma \ref{t3.3}
yield the condition
\begin{eqnarray*}
\Lambda_{[1]}=\Lambda\cap S_1 = (L\cap S_1): L^{[0]}
= L_{[1]} : K^i.
\end{eqnarray*}
This condition indeed bounds $L_{[1]}=L\cap S_1$ from above,
as we will show now. The argument follows \cite[$5^\circ$]{Fa67}.

We apply Lemma \ref{t3.3} and Lemma \ref{t2.3} (b) with
$$ (V_1,V_2,V_3,\beta)_{\textup{in Lemma \ref{t2.3}}}
\sim (R^{[0]},S_{[1]},S_{[1]},
\textup{multiplication})_{\textup{here}}$$
and write also $\beta^*$ as multiplication. Then
$$\Lambda_{[1]}^{*\Z} = L_{[1]}^{*\Z}\cdot K^i.$$
We can multiply both sides with the maximal order
$\Lambda_{max}(F)$. Then $K^i\cdot \Lambda_{max}(F)$ is a
$\Lambda_{max}(F)$-ideal, so it is invertible by
Corollary \ref{t6.2}. We obtain
\begin{eqnarray*}
\Lambda_{[1]}^{*\Z}\cdot \Lambda_{max}(F) &=& 
L_{[1]}^{*\Z}\cdot \Lambda_{max}(F)\cdot 
(K^i\cdot \Lambda_{max}(F)),\\
\Lambda_{[1]}^{*\Z}\cdot \Lambda_{max}(F)\cdot
(K^i\cdot \Lambda_{max}(F))^{-1} &=& 
L_{[1]}^{*\Z}\cdot \Lambda_{max}(F).
\end{eqnarray*}
There is a natural number $r$ with 
$r\cdot\Lambda_{max}(F)\subset\Lambda^{[0]}$ 
by Lemma \ref{t2.2} (c). Multiplying both sides by $r$ gives
\begin{eqnarray*}
r\cdot \Lambda_{[1]}^{*\Z}\cdot\Lambda_{max}(F)\cdot
(K^i\cdot\Lambda_{max}(F))^{-1} &=& 
L_{[1]}^{*\Z}\cdot r\cdot \Lambda_{max}(F)\\
&\subset& L_{[1]}^{*\Z}\cdot \Lambda^{[0]} = L_{[1]}^{*\Z},\\
\Bigl(r\cdot \Lambda_{[1]}^{*\Z}\cdot\Lambda_{max}(F)\cdot
(K^i\cdot\Lambda_{max}(F))^{-1}\Bigr)^{*\Z}
&\supset& L_{[1]}.
\end{eqnarray*}
This bounds $L_{[1]}$ from above. $L_{[1]}$ is bounded from 
below by $L_{[1]}\supset \Lambda_{[1]}$ because
$L\supset \Lambda$. By Lemma \ref{t2.2} (d) the set 
in \eqref{6.4} is finite.

{\bf Step 3:}
We will show for each $i\in I$ inductively for 
$m\in\{1,2,...,n_{max}+1\}$ that the set of tuples
\begin{eqnarray}\label{6.5}
\{(L_{[m]},L_{[m-1]},...,L_{[1]})&|& L\in\LL(A)\textup{ with }
\OO(L)=\Lambda,\\ 
&&L^{[0]}=K^i,\ L\supset \Lambda\}\nonumber
\end{eqnarray}
is finite. Step 2 shows this for $m=1$. Suppose it is true for
some $m\leq n_{max}$. We fix
$(L_{[m]},L_{[m-1]},...,L_{[1]})=(K_m,K_{m-1},...,K_1)$ and want
to show that the set
\begin{eqnarray}\label{6.6}
\{L_{[m+1]}&|& L\in\LL(A)\textup{ with }\OO(L)=\Lambda,\ 
L^{[0]}=K^i,\ L\supset \Lambda,\\
&&(L_{[m]},...,L_{[1]})=(K_m,...,K_1)\}\nonumber
\end{eqnarray}
is finite. 

We apply Lemma \ref{t3.3} and Lemma \ref{t2.3} (a) with
$$(V_1,V_2,V_3,\beta)_{\textup{in Lemma \ref{t2.3}}}
\sim (R^{[1]},S_{[m+1]},S_{[m]},\textup{multiplication})_{\textup{here}}.$$
Then
$$\Lambda^{[1]}\cdot L_{[m+1]}\subset L_{[m]}=K_m,$$
so
$$L_{[m+1]}\subset K_m:\Lambda^{[1]}.$$
$L_{[m+1]}$ is therefore bounded from above by the full lattice
$K_m:\Lambda^{[1]}$ in $S_{[m+1]}$.
It is bounded from below by $L_{[m+1]}\supset \Lambda_{[m+1]}$
because $L\supset \Lambda$. 
Therefore the set in \eqref{6.6} is finite.

The set in \eqref{6.5} is finite for any $m$, especially for
$m=n_{max}+1$. 

We conclude with Lemma \ref{t2.4} (c) with $V=A$ and 
$V_\bullet=S_\bullet$ that the set in 
\eqref{6.3} is finite.\hfill$\Box$

\section{Localization}\label{s7}
\setcounter{equation}{0}
\setcounter{table}{0}

In \cite[Ch. I \S 12]{Ne99} orders in algebraic number fields
are studied, and good use is made of localizations by prime
ideals.
For orders in an arbitrary algebra $A$ as in Theorem \ref{t3.1}
another family of localizations is proposed in \cite{Fa64}
and \cite[$6^\circ$]{Fa65}. In this section we cite
some basic results on this family of localizations and prove
the new Theorem \ref{t7.6}. In the next section we apply
these tools to obtain results similar to some results
in \cite[Ch. I \S 12]{Ne99}.

\begin{definition}\label{t7.1}
Recall that $\P\subset\N$ denotes the set of all prime numbers
and that for $p\in\P$ $\F_p:=\Z/p\Z$ denotes the field with
$p$ elements. 

(a) For $p\in\P$ denote (this notation follows
\cite[Ch. I \S 11]{Ne99})
\begin{eqnarray*}
\Z_{(p)}&:=& \{\frac{a}{b}\,|\, a\in\Z,b\in\N,\gcd(a,b)=1,
\gcd(p,b)=1\}\subset\Q.
\end{eqnarray*}
It is a subring of $\Q$ and a local ring with maximal ideal
$p\Z_{(p)}$ and group of units
$\Z_{(p)}^{unit}=\Z_{(p)}-p\Z_{(p)}$. 

(b) Let $V$ be an $n$-dimensional $\Q$-vector space for some
$n\in\N$. Let $L$ be a full lattice in $V$. For $p\in\P$ denote
\begin{eqnarray*}
L_{(p)}&:=& \Z_{(p)}\cdot L:=\{qa\,|\, q\in\Z_{(p)},a\in L\}
\subset V.
\end{eqnarray*}
It is a free $\Z_{(p)}$-module of rank $n$ with 
$L\subset L_{(p)}\subset \Q\cdot L_{(p)}=V$.
\end{definition}

The following basic facts are stated in \cite[$6^\circ$]{Fa65}.
Part (c) is proved in \cite[$3^\circ$]{Fa64}.

\begin{theorem}\label{t7.2}
Let $V$ be an $n$-dimensional $\Q$-vector space for some 
$n\in\N$. Let $L_1,L_2,L_3$ and $L_4$ be full lattices in $V$
with $L_3\supset L_4$.

(a)
\begin{eqnarray*}
L_1 &=& \bigcap_{p\in\P} (L_1)_{(p)}
\end{eqnarray*}
and for each $p\in\P$ 
\begin{eqnarray*}
(L_1+L_2)_{(p)}&=& (L_1)_{(p)}+(L_2)_{(p)},\\
(L_1\cap L_2)_{(p)}&=& (L_1)_{(p)}\cap (L_2)_{(p)},\\
(L_1\cdot L_2)_{(p)}&=& (L_1)_{(p)}\cdot (L_2)_{(p)}
\quad\textup{if }V=A\textup{ as in Theorem \ref{t3.1},}\\
(L_1: L_2)_{(p)}&=& (L_1)_{(p)}: (L_2)_{(p)}
\quad\textup{if }V=A\textup{ as in Theorem \ref{t3.1}.}
\end{eqnarray*}

(b) The subset of $\P$ of elements $p$ with
$$(L_1)_{(p)}\neq (L_2)_{(p)}$$
is finite.

(c) \cite[Proposition 6.1]{Fa65} \cite[$3^\circ$ Lemma]{Fa64}
For each $p\in\P$ let $U_p\subset V$ be a free 
$\Z_{(p)}$-module of rank $n$. Suppose that
$$U_p=(L_1)_{(p)}\quad\textup{for all except finitely many }
p\in\P.$$
Then there is a unique full lattice $L$ in $V$ with
$$L_{(p)}=U_p\quad\textup{for each }p\in\P.$$
By part (a) it is $L=\bigcap_{p\in\P}U_p.$

(d) There is a natural isomorphism
\begin{eqnarray}\label{7.1}
L_3/L_4&\cong& \bigoplus_{p\in\P:\ (L_3)_{(p)}\neq (L_4)_{(p)}}
(L_3)_{(p)}/(L_4)_{(p)}
\end{eqnarray}
as torsion $\Z$-modules respectively finite additive groups.
Observe that by part (b) the set $\{p\in\P\,|\, 
(L_3)_{(p)}\neq (L_4)_{(p)}\}$ is finite.
\end{theorem}

The following result gives a remarkabe relation between 
invertibility and localization if $V=A$ is as in Theorem 
\ref{t3.1}.

\begin{theorem}\label{t7.3}
\cite[Proposition 27.1]{Fa65}
Let $A$ be as in Theorem \ref{t3.1}.
Let the full lattice $L\in \LL(A)$ have order $\Lambda:=\OO(L)$. Then:
\begin{eqnarray}\label{7.2}
L\textup{ is invertible}\iff 
L_{(p)}\textup{ is a principal }\Lambda_{(p)}
\textup{-ideal for each }p\in\P,\\
\textup{i.e. }L_{(p)}=a_p\Lambda_{(p)}
\textup{ for some }a_p\in A^{unit}
\textup{ for each }p\in\P,\nonumber\\
\textup{with }a_p=1_A\textup{ for all except finitely many }
p.\nonumber
\end{eqnarray}
\end{theorem}

In fact, in \cite[$27^\circ$]{Fa65} only the more difficult
implication $\Rightarrow$ is proved.

Proof of the implication $\Leftarrow$:
Apply Theorem \ref{t7.2} (c) to 
$U_p:=a_p^{-1}\cdot \Lambda_{(p)}$ for each $p\in\P$.
It gives a full lattice $\www{L}$ with
$(\www{L})_{(p)}=a_p^{-1}\cdot\Lambda_{(p)}$ for each
$p\in\P$, so
\begin{eqnarray*}
(\www{L}\cdot L)_{(p)}=\Lambda_{(p)}\quad
\textup{for each }p\in\P,\quad\textup{so }
\www{L}\cdot L=\Lambda,\\
(\Lambda\cdot\www{L})_{(p)}=a_p^{-1}\Lambda_{(p)}=(\www{L})_{(p)}
\textup{ for each }p\in\P,\quad\textup{so }\Lambda\cdot\www{L}=\www{L},
\end{eqnarray*}
so $L$ is invertible with $L^{-1}=\www{L}$.
\hfill$\Box$

\begin{remarks}\label{t7.4}
Let $A$ be as in Theorem \ref{t3.1}.
By Theorem \ref{t5.8} (a) two full lattices $L_1$ and $L_2$
are $w$-equivalent if and only if $\OO(L_1)=\OO(L_2)$
and an invertible lattice $L_3\in G(\OO(L_1))$ with 
$L_1L_3=L_2$ exists (in fact, then $L_3=L_2:L_1$ is unique).
By Theorem \ref{t7.3} this is equivalent to
\begin{eqnarray}\label{7.3}
(L_2)_{(p)} &=& a_p\cdot (L_1)_{(p)}\quad\textup{for each }
p\in\P
\end{eqnarray}
with $a_p=1_A$ for almost all $p\in\P$ and $a_p\in A^{unit}$
for the other $p\in\P$. Lattices $L_1$ and $L_2$ with 
\eqref{7.3} are called {\it locally equivalent}.
So, in our commutative situation $w$-equivalence and
local equivalence coincide
\cite[$28^\circ$]{Fa65}. 
\end{remarks}

The following characterization of elements of
$\Lambda_{(p)}^{unit}:=(\Lambda_{(p)})^{unit}$ 
will be used in the proof of Theorem \ref{t7.6}.

\begin{lemma}\label{t7.5}
Let $A$ be as in Theorem \ref{t3.1}. Let $\Lambda$ be an order
in $A$. Let $p\in\P$ be a prime number. For any $a\in A$ 
$\det(a)\in\Q$ denotes the determinant of the endomorphism
$(\mu_a:A\to A,\ b\mapsto ab)$ of $A$.

(a) Consider an element $a\in \Lambda_{(p)}$. Then
\begin{eqnarray*}
a\in \Lambda_{(p)}^{unit}&\iff& 
\det(a)\in\Z_{(p)}^{unit}.
\end{eqnarray*}

(b) Consider an element $a\in\Lambda$. Then
\begin{eqnarray*}
a\in \Lambda_{(p)}^{unit}&\iff&
a+p\Lambda\in \Lambda/p\Lambda\textup{ is in }
(\Lambda/p\Lambda)^{unit}.
\end{eqnarray*}
\end{lemma}

{\bf Proof:}
(a) $\Rightarrow$: $a\in\Lambda_{(p)}$ implies 
$\det(a)\in\Z_{(p)}$, and $a^{-1}\in \Lambda_{(p)}$ implies
$(\det(a))^{-1}=\det(a^{-1})\in\Z_{(p)}$.

$\Leftarrow$: Suppose $\det(a)\in\Z_{(p)}^{unit}$. 
Denote by $p_{Ch,a}(t)\in\Z[t]$
the characteristic polynomial of the endomorphism
$(\mu_a:A\to A,\ b\mapsto ab)$. Then 
$p_{Ch,a}(a)=0$ by Cayley-Hamilton, and 
$t$ divides $p_{Ch,a}(t)-(-1)^n\det(a)$, so 
\begin{eqnarray*}
b&:=&\frac{p_{Ch,a}(a)-(-1)^n\det(a)}{a}=\frac{(-1)^{n+1}\det(a)}{a}
\in\Z[a]\subset\Lambda_{(p)},\\
a^{-1}&=& (-1)^{n+1}(\det(a))^{-1}\cdot b\in\Lambda_{(p)}.
\end{eqnarray*}

(b) Suppose $a\in\Lambda$. Then $\det(a)\in\Z$.
Then by part (a) 
$a\in\Lambda_{(p)}^{unit}\Leftrightarrow p\nmid \det(a)$.
This is equivalent to $a+p\Lambda\in (\Lambda/p\Lambda)^{unit}$.
\hfill$\Box$

\bigskip
In section \ref{s8} we will use Theorem \ref{t7.6}.
The crucial point is the surjectivity of the group homomorphism
in \eqref{7.6}.

\begin{theorem}\label{t7.6}
Let $A$ be as in Theorem \ref{t3.1}.
Let $\Lambda$ be an order, and let $L\subset\Lambda$
be a $\Lambda$-ideal (not necessarily exact).
The set $P_0:=\{p\in\P\,|\, \Lambda_{(p)}\neq L_{(p)}\}$
is finite by Theorem \ref{t7.2} (b). 

(a) The quotient $\Lambda/L$ is a finite commutative ring
with unit element. For each $p\in P_0$
the quotient $\Lambda_{(p)}/L_{(p)}$ is a finite commutative
ring with unit element. There is a natural isomorphism
\begin{eqnarray}\label{7.4}
\Lambda/L&\to& \prod_{p\in P_0}
\Lambda_{(p)}/L_{(p)}
\end{eqnarray}
of finite commutative rings with unit elements.
It induces an isomorphism
\begin{eqnarray}\label{7.5}
(\Lambda/L)^{unit}&\to& \prod_{p\in P_0}
(\Lambda_{(p)}/L_{(p)})^{unit}
\end{eqnarray}
of finite commutative multiplicative groups
(of units in the rings).

(b) Fix $p\in P_0$ and an element $a_1\in\Lambda_{(p)}$ with $a_1+L_{(p)}\in (\Lambda_{(p)}/L_{(p)})^{unit}$. 
There is an element 
$a_3\in\Lambda\cap \Lambda_{(p)}^{unit}$ with 
\begin{eqnarray*}
a_3+L&\in& (\Lambda/L)^{unit},\\
a_3+L_{(p)}&=& a_1+L_{(p)},\\
a_3+L_{(q)}&=& 1_A+L_{(q)}\quad\textup{for }q\in P_0-\{p\},
\end{eqnarray*}
so the image of $a_3+L\in (\Lambda/L)^{unit}$ under the 
isomorphism in \eqref{7.5} is $a_1+L_{(p)}$ in the factor
$(\Lambda_{(p)}/L_{(p)})^{unit}$ and the unit element in 
each other factor $(\Lambda_{(q)}/L_{(q)})^{unit}$.
Especially, the group homomorphism
\begin{eqnarray}\label{7.6}
\Lambda_{(p)}^{unit}&\to& (\Lambda_{(p)}/L_{(p)})^{unit}.
\end{eqnarray}
is surjective.
\end{theorem}

{\bf Proof:}
(a) By Theorem \ref{7.2} (d) there is a natural map as in 
\eqref{7.4} which is an isomorphism between finite additive
groups. Obviously for $a,b\in L$
\begin{eqnarray*}
\Lambda/L\owns (a+L)(b+L)\mapsto
\prod_{p\in P_0}\Bigl((a+L_{(p)})(b+L_{(p)})\Bigr)\\
=\Bigl(\prod_{p\in P_0}(a+L_{(p)})\Bigr)\cdot
\Bigl(\prod_{p\in P_0}(b+L_{(p)})\Bigr)
\in\prod_{p\in P_0}\Lambda_{(p)}/L_{(p)},
\end{eqnarray*}
so it is an isomorphism between finite commutative rings
with unit elements.

(b) Consider an element $a_1\in \Lambda_{(p)}$ with 
$a_1+L_{(p)}\in (\Lambda_{(p)}/L_{(p)})^{unit}$. 
Because of the isomorphism in \eqref{7.5} we can choose
an element $a_2\in\Lambda$ with 
\begin{eqnarray*}
a_2+L&\in& (\Lambda/L)^{unit},\\
a_2+L_{(p)}&=& a_1+L_{(p)}\quad\textup{and}\\
a_2+L_{(q)}&=& 1_A+L_{(q)}\textup{ for }q\in  P_0-\{p\}.
\end{eqnarray*}
We want to find an element $a_3\in a_2+L(\subset\Lambda)$
with $a_3\in \Lambda_{(p)}^{unit}$. 
This is surprisingly nontrivial. 

We will use Lemma \ref{t7.5} (b) and the quotient ring
$\Lambda/p\Lambda$. We have to discuss its structure.
The quotient $\Lambda/p\Lambda$ is a finite dimensional
commutative $\F_p$-algebra with unit element.
The finite field $\F_p=\Z/p\Z$
is a perfect field. By Remark \ref{t3.2} (i), 
Theorem \ref{t3.1} applies to $\Lambda/p\Lambda$  
instead of $A$ and the field $\F_p$ instead of the
field $\Q$. We write the decompositions
in Theorem \ref{t3.1} for $\Lambda/p\Lambda$ as
\begin{eqnarray}\label{7.7}
\Lambda/p\Lambda &=& 
\bigoplus_{j=1}^{\www{k}} A_{\Lambda,p}^{(j)},\qquad
A_{\Lambda,p}^{(j)}\ =\ F_{\Lambda,p}^{(j)}\oplus
N_{\Lambda,p}^{(j)}.
\end{eqnarray}
Here $\www{k}\in\N$, 
$A_{\Lambda,p}^{(1)},...,A_{\Lambda,p}^{\www{k}}$ are 
irreducible and local $\F_p$-algebras with
$A_{\Lambda,p}^{(i)}\cdot A_{\Lambda,p}^{(j)}=\{0\}$ for
$i\neq j$. $N_{\Lambda,p}^{(j)}$ is the maximal ideal
in $A_{\Lambda,p}^{(j)}$ and consists of nilpotent elements.
$F_{\Lambda,p}^{(j)}$ is an $\F_p$-subalgebra of
$A_{\Lambda,p}^{(j)}$ and a finite field extension of $\F_p$.
$A_{\Lambda,p}^{(j)}$ is an $F_{\Lambda,p}^{(j)}$-algebra.

An element $b=\sum_{j=1}^{\www{k}}(c_j+n_j)\in \Lambda/p\Lambda$
with $c_j\in F_{\Lambda,p}^{(j)}$ and 
$n_j\in N_{\Lambda,p}^{(j)}$ is a unit if and only
if each $c_j\neq 0$. 

The $\Lambda$-ideal $L$ induces an ideal 
$\www{L}:=(L+p\Lambda)/p\Lambda\subset\Lambda/p\Lambda$
in the ring $\Lambda/p\Lambda$. It satisfies
$$\www{L}= \bigoplus_{j=1}^{\www{k}} \www{L}^{(j)}\quad
\textup{where }\www{L}^{(j)}:=\www{L}\cap A_{\Lambda,p}^{(j)},$$
and $\www{L}^{(j)}$ is an ideal in the ring 
$A_{\Lambda,p}^{(j)}$. Either $\www{L}^{(j)}$
is a proper ideal, then 
$\www{L}^{(j)}\subset N_{\Lambda,p}^{(j)}$,
or $\www{L}^{(j)}=A_{\Lambda,p}^{(j)}$. 
Denote by $J_1\subset\{1,...,\www{k}\}$ the set of indices $j$
with $\www{L}^{(j)}=A_{\Lambda,p}^{(j)}$. 

Denote by $1_{A_{\Lambda,p}^{(j)}}$ the unit element in the ring
$A_{\Lambda,p}^{(j)}$. For $j\in J_1$ it is in 
$\www{L}^{(j)}\subset \www{L}$. 
For $j\in J_1$ choose an element $e_j\in L$ which maps
under the surjective map
$$L\to \www{L}=\bigoplus_{j=1}^{\www{k}}\www{L}^{(j)}$$
to $1_{A_{\Lambda,p}^{(j)}}$. 

Consider the element $a_2\in \Lambda$ chosen above.
Because $a_2+L\in (\Lambda/L)^{unit}$, an element
$\www{a}_2\in\Lambda$ with 
$a_2\www{a}_2+L=(a_2+L)(\www{a}_2+L)=1_A+L$ can be chosen.
Write the images of $a_2$ and $\www{a}_2$ in $\Lambda/p\Lambda$ as
\begin{eqnarray*}
a_2+p\Lambda=\sum_{j=1}^{\www{k}}(c_j+n_j)\quad\textup{and}\quad
\www{a}_2+p\Lambda=\sum_{j=1}^{\www{k}}(\www{c}_j+\www{n}_j)\\
\textup{with}\quad c_j,\www{c}_j\in F_{\Lambda,p}^{(j)}
\quad\textup{and}\quad n_j,\www{n}_j\in N_{\Lambda,p}^{(j)}.
\end{eqnarray*}
Then
\begin{eqnarray*}
a_2\www{a}_2+p\Lambda =\sum_{j=1}^{\www{k}}\Bigl(
c_j\www{c}_j+(c_j\www{n}_j+n_j\www{c}_j+n_j\www{n}_j)\Bigr)\\
\textup{with }c_j\www{c}_j\in F_{\Lambda,p}^{(j)}
\quad\textup{and}\quad 
c_j\www{n}_j+n_j\www{c}_j+n_j\www{n}_j\in N_{\Lambda,p}^{(j)}.
\end{eqnarray*}
The equality $a_2\www{a}_2+L=1_A+L$ induces in 
$\Lambda/p\Lambda$ the inclusion
\begin{eqnarray*}
a_2\www{a}_2+p\Lambda \in 
\Bigl(\sum_{j=1}^{\www{k}}1_{A_{\Lambda,p}^{(j)}}\Bigr)+\www{L}
=\sum_{j=1}^{\www{k}}\Bigl(
1_{A_{\Lambda,p}^{(j)}}+\www{L}^{(j)}\Bigr).
\end{eqnarray*}
Therefore for $j\in\{1,...,\www{k}\}-J_1$
$$c_j\www{c}_j= 1_{A_{\Lambda,p}^{(j)}},\quad \textup{so}\quad
c_j\in F_{\Lambda,p}^{(j)}-\{0\}.$$
Define by $J_2\subset J_1$ the set of indices $j$ with 
$c_j=0$. Define
$$a_3:=a_2+\sum_{j\in J_2}e_j\in a_2+L\subset \Lambda.$$
Then $a_3+p\Lambda\in \Lambda/p\Lambda$ has for each
$j\in\{1,...,\www{k}\}$ a nonvanishing part $c_j$ or 
$1_{A_{\Lambda,p}^{(j)}}$
in $F_{\Lambda,p}^{(j)}$. Therefore 
$a_3+p\Lambda\in (\Lambda/p\Lambda)^{unit}$,
so by Lemma \ref{7.5} (b) $a_3\in \Lambda_{(p)}^{unit}$.
\hfill$\Box$

\begin{remarks}\label{t7.7} If in the proof of Theorem \ref{t7.6}
$L\subset p\Lambda$ then $J_1=\emptyset$ and $a_3=a_2$ in the proof.
Then any preimage in $\Lambda$ of a class in 
$(\Lambda/L)^{unit}$ is in $\Lambda_{(p)}^{unit}$.
The inclusion $L\subset p\Lambda$ holds often, but not always.
\end{remarks}

\section{Exact sequences for an order and a smaller order}
\label{s8}
\setcounter{equation}{0}
\setcounter{table}{0}

Throughout the whole paper $A$ is as in Theorem \ref{t3.1}, 
so $A$ is a finite dimensional commutative $\Q$-algebra 
with unit element $1_A$.

Formula \eqref{5.3} can be improved if there 
$\OO(L_1)\supset \OO(L_2)$ and $L_2$ is invertible. 
This is subject of Lemma \ref{t8.1} (a).

\begin{lemma}\label{t8.1}

(a) Let $L_1,L_2\in\LL(A)$ with $\OO(L_1)\supset \OO(L_2)$
and $L_2$ invertible. Then 
\begin{eqnarray*}
\OO(L_1L_2)=\OO(L_1),\quad L_1:L_2=L_1L_2^{-1},\quad 
\OO(L_1:L_2)=\OO(L_1).
\end{eqnarray*}

(b) Let $\Lambda_1$ and $\Lambda_2\in\LL(A)$ be two orders 
with $\Lambda_2\subsetneqq \Lambda_1$. 
If $L\in G(\Lambda_2)$ then $\Lambda_1L\in G(\Lambda_1)$.
This leads to group homomorphisms
\begin{eqnarray}
G(\Lambda_2)\to G(\Lambda_1), &&L\mapsto \Lambda_1 L,
\label{8.1}\\
G([\Lambda_2]_\varepsilon)\to G([\Lambda_1]_\varepsilon),&& 
[L]_\varepsilon \mapsto [\Lambda_1 L]_\varepsilon.\label{8.2}
\end{eqnarray}
\end{lemma}

{\bf Proof:}
(a) Suppose $a\in\OO(L_1L_2)$. Then 
\begin{eqnarray*}
aL_1=aL_1\OO(L_2)=aL_1L_2L_2^{-1}\subset L_1L_2L_2^{-1}
=L_1\OO(L_2)=L_1,
\end{eqnarray*}
so $a\in\OO(L_1)$, so $\OO(L_1L_2)\subset\OO(L_1)$. 
Equality follows with  \eqref{5.3}. 

By definition of $L_1:L_2$, we have $(L_1:L_2)L_2\subset L_1$.
This gives $\subset$ in the next formula, the first 
$=$ comes from \eqref{5.3},
\begin{eqnarray*}
L_1:L_2 = (L_1:L_2)\OO(L_2)=(L_1:L_2)L_2L_2^{-1}\subset
L_1L_2^{-1}.
\end{eqnarray*}
$L_1:L_2\supset L_1L_2^{-1}$ is true because $\OO(L_2)\subset
\OO(L_1)$. We obtain $L_1:L_2=L_1L_2^{-1}$. 
Finally $\OO(L_1L_2^{-1})=\OO(L_1)$ because $L_2^{-1}$ is
invertible with $\OO(L_2^{-1})=\OO(L_2)\subset \OO(L_1)$. 

(b) $(\Lambda_1 L)L^{-1}=\Lambda_1\Lambda_2=\Lambda_1$ and 
$\Lambda_1(\Lambda_1 L)=\Lambda_1 L$ show that $\Lambda_1 L$ is 
invertible with $\OO(\Lambda_1 L)=\Lambda_1$.
Therefore the maps in \eqref{8.1} and \eqref{8.2} are well
defined. Obviously they are group homomorphisms. \hfill$\Box$

\bigskip
Theorem \ref{t8.2} treats the group homomorphisms in 
\eqref{8.1} and \eqref{8.2}. 
They turn out to be surjective. They are part of the
exact sequences in \eqref{8.3} and \eqref{8.7}. Theorem \ref{t8.2}
uses the localizations $\Z\dashrightarrow\Z_{(p)}$ in 
section \ref{s7}. It has similarities to and generalizes results
in \cite[Ch. I \S 12]{Ne99}. The Remarks \ref{t8.3}
say more on this.

\begin{theorem}\label{t8.2}
Let $\Lambda_1$ and $\Lambda_2$ be two orders in $A$ with
$\Lambda_2\subsetneqq\Lambda_1$. The full lattice 
$C:=\Lambda_2:\Lambda_1$ is called {\sf conductor} of the
pair $(\Lambda_1,\Lambda_2)$.
The set $P_0:=\{p\in\P\,|\, (\Lambda_1)_{(p)}
\neq (\Lambda_2)_{(p)}\}$ is finite by Theorem \ref{t7.2} (b).

(a) $\OO(C)\supset\Lambda_1$. The conductor $C$ is the biggest
$\Lambda_1$-ideal in $\Lambda_2$.

(b) The following sequence is exact,
\begin{eqnarray}\label{8.3}
1\to \prod_{p\in P_0}(\Lambda_2)_{(p)}^{unit} \to 
\prod_{p\in P_0}(\Lambda_1)_{(p)}^{unit} \to G(\Lambda_2)
\to G(\Lambda_1)\to 1.
\end{eqnarray}
Especially, the group homomorphism in \eqref{8.1} is surjective.
Here the image $L$ in $G(\Lambda_2)$ of a tuple 
$(a_p)_{p\in P_0}\in \prod_{p\in P_0}(\Lambda_1)_{(p)}^{unit}$
is given as follows. 
Define $a_p:=1_A$ for $p\in\P-P_0$. By Theorem \ref{t7.2} (c)
there is a unique full lattice $L$ in $A$ with
$L_{(p)}=a_p(\Lambda_2)_{(p)}$ for each $p\in\P$, namely
$L=\bigcap_{p\in\P}a_p(\Lambda_2)_{(p)}$. 

(c) There is a natural isomorphism of groups
\begin{eqnarray}\label{8.4}
(\Lambda_1/C)^{unit}/(\Lambda_2/C)^{unit}&\to& 
\bigoplus_{p\in P_0}(\Lambda_1)_{(p)}^{unit}/
(\Lambda_2)_{(p)}^{unit}.
\end{eqnarray}

(d) The isomorphism of groups
\begin{eqnarray}\label{8.5}
(\Lambda_1/C)^{unit}/(\Lambda_2/C)^{unit} &\to& 
\ker\bigl(G(\Lambda_2)\to G(\Lambda_1)\bigr)
\end{eqnarray}
which results from (b) and (c) is given by
\begin{eqnarray}\label{8.6}
((a+C)\mmod (\Lambda_2/C)^{unit})&\mapsto& C+a\Lambda_2,
\end{eqnarray}
where $a\in\Lambda_1$, $a+C\in (\Lambda_1/C)^{unit}
\subset \Lambda_1/C$. 

(e) The maps in the following sequence are the natural ones,
the sequence is exact,
\begin{eqnarray}\label{8.7}
1\to \Lambda_2^{unit}\to \Lambda_1^{unit}\to 
\prod_{p\in P_0}\frac{(\Lambda_1)_{(p)}^{unit}}
{(\Lambda_2)_{(p)}^{unit}}
\to G([\Lambda_2]_\varepsilon)\to G([\Lambda_1]_\varepsilon)
\to 1.
\end{eqnarray}
Especially, the group homomorphism in \eqref{8.2} is surjective.

(f) The group $\Lambda_2^{unit}$ has finite index in the group
$\Lambda_1^{unit}$.
By Theorem \ref{t6.5} the groups $G([\Lambda_2]_\varepsilon)$
and $G[\Lambda_1]_\varepsilon)$ are finite. The size of one of them
can be calculated by the size of the other one with the following
formula, 
\begin{eqnarray}\label{8.8}
\frac{|G([\Lambda_2]_\varepsilon)|}{|G([\Lambda_1]_\varepsilon)|}
&=& \frac{|(\Lambda_1/C)^{unit}|}{|(\Lambda_2/C)^{unit}|}
\cdot \frac{1}{[\Lambda_1^{unit}:\Lambda_2^{unit}]}.
\end{eqnarray}
\end{theorem}

{\bf Proof:}
(a) $\OO(C)\supset\Lambda_1$ follows from \eqref{5.3}.
So $C$ is a $\Lambda_1$-ideal. 
Any $\Lambda_1$-ideal in $\Lambda_2$ is contained in
$\Lambda_2:\Lambda_1=C$. On the other hand,  
$1_A\in \Lambda_1$ implies 
$C=C\cdot 1_A\subset C\cdot \Lambda_1\subset\Lambda_2$.
So $C$ is a $\Lambda_1$-ideal in $\Lambda_2$. 
Therefore $C$ is the unique biggest $\Lambda_1$-ideal in 
$\Lambda_2$.

(b) We carry out the proof in four steps. 

Step 1:
In step 1 we show that the map in \eqref{8.1} is surjective. 
Let $L_1\in G(\Lambda_1)$. For almost all $p\in\P$
$(L_1)_{(p)}=(\Lambda_1)_{(p)}$. Define $b_p:=1_A$ for these
$p$. For the finitely many $p\in\P$ with 
$(L_1)_{(p)}\neq (\Lambda_1)_{(p)}$ we have  
$(L_1)_{(p)}=b_p (\Lambda_1)_{(p)}$ for some
$b_p\in A^{unit}$ by Theorem \ref{t7.3} and by 
$L_1\in G(\Lambda_1)$.
By Theorem \ref{t7.2} (c) there is a unique full lattice 
$L_2$ with $(L_2)_{(p)}=b_p(\Lambda_2)_{(p)}$ for each
$p\in\P$, namely $L_2=\bigcap_{p\in\P}b_p(\Lambda_2)_{(p)}$.
By Theorem \ref{t7.3} it is in $G(\Lambda_2)$. Now
\begin{eqnarray*}
(\Lambda_1L_2)_{(p)}=a_p(\Lambda_1)_{(p)}(\Lambda_2)_{(p)}
=a_p(\Lambda_1)_{(p)}=(L_1)_{(p)}
\quad\textup{for each }p\in\P,
\end{eqnarray*}
so $\Lambda_1L_2=L_1$. The map in \eqref{8.1} is surjective.

Step 2: Given a tuple $(a_p)_{p\in P_0}\in \prod_{p\in P_0}
(\Lambda_1)_{(p)}^{unit}$, define for each $p\in \P-P_0$
$a_p:=1_A$ and consider the full lattice 
$L:=\bigcap_{p\in\P}a_p(\Lambda_2)_{(p)}$, which exists by
Theorem \ref{t7.2} (c). 
Here we show $L\in\ker(G(\Lambda_2)\to G(\Lambda_1))$. 
By Theorem \ref{t7.3} $L\in G(\Lambda_2)$. Observe
\begin{eqnarray*}
(\Lambda_1L)_{(p)} = a_p(\Lambda_1\Lambda_2)_{(p)}
=a_p(\Lambda_1)_{(p)}\stackrel{(*)}{=}(\Lambda_1)_{(p)}\quad\textup{for each }p\in\P.
\end{eqnarray*} 
Here $\stackrel{(*)}{=}$ follows
from $a_p\in (\Lambda_1)_{(p)}^{unit}$ for $p\in P_0$ and
$a_p=1_A$ for $p\in\P-P_0$. The observation shows
$\Lambda_1L=\Lambda_1$, 
so $L\in\ker(G(\Lambda_2)\to G(\Lambda_1))$. 

Step 3: Vice versa, consider a full lattice
$L_3\in\ker(G(\Lambda_2)\to G(\Lambda_1))$. For $p\in\P-P_0$
\begin{eqnarray*}
(L_3)_{(p)}&=&(\Lambda_2L_3)_{(p)}
=(\Lambda_2)_{(p)}(L_3)_{(p)}\\
&=&(\Lambda_1)_{(p)}(L_3)_{(p)}
=(\Lambda_1L_3)_{(p)}
=(\Lambda_1)_{(p)}
=(\Lambda_2)_{(p)}.
\end{eqnarray*}
For $p\in P_0$ by Theorem \ref{t7.3} and by $L_3\in G(\Lambda_2)$ 
there exists 
$c_p\in A^{unit}$ with $(L_3)_{(p)}=c_p(\Lambda_2)_{(p)}$.
For $p\in P_0$ 
\begin{eqnarray*}
(\Lambda_1)_{(p)}
=(\Lambda_1L_3)_{(p)}=(\Lambda_1)_{(p)}c_p(\Lambda_2)_{(p)}
=c_p(\Lambda_1)_{(p)},\\
\textup{so }c_p\in (\Lambda_1)_{(p)}^{unit}.
\end{eqnarray*}
$L_3$ is in the image of the tuple $(c_p)_{p\in P_0}
\in \prod_{p\in P_0}(\Lambda_1)_{(p)}^{unit}$.

Step 4: The kernel of the map 
$\prod_{p\in P_0}(\Lambda_1)_{(p)}^{unit}\to G(\Lambda_2)$
is $\prod_{p\in P_0}(\Lambda_2)_{(p)}^{unit}$ because
$a_p(\Lambda_2)_{(p)}=(\Lambda_2)_{(p)}$ is equivalent to 
$a_p\in (\Lambda_2)_{(p)}^{unit}$.

(c) The isomorphism in \eqref{7.5} in Theorem \ref{t7.6} (a)
holds for $(\Lambda,L)=(\Lambda_1,C)$ and for 
$(\Lambda,L)=(\Lambda_2,C)$. It induces a group isomorphism
\begin{eqnarray}\nonumber
\frac{(\Lambda_1/C)^{unit}}{(\Lambda_2/C)^{unit}} &\to& 
\Bigl(\prod_{p\in P_0}((\Lambda_1)_{(p)}/C_{(p)})^{unit}\Bigr)
/\Bigl(\prod_{p\in P_0}((\Lambda_2)_{(p)}/C_{(p)})^{unit}\Bigr)\\
&\cong& \prod_{p\in P_0}
\frac{((\Lambda_1)_{(p)}/C_{(p)})^{unit}}
{((\Lambda_2)_{(p)}/C_{(p)})^{unit}}.\label{8.9}
\end{eqnarray}
It remains to show for each $p\in P_0$ that there is a 
natural isomorphism of groups
\begin{eqnarray}\label{8.10}
\frac{(\Lambda_1)_{(p)}^{unit}}{(\Lambda_2)_{(p)}^{unit}}
&\to& \frac{((\Lambda_1)_{(p)}/C_{(p)})^{unit}}
{((\Lambda_2)_{(p)}/C_{(p)})^{unit}}.
\end{eqnarray}
By Theorem \ref{t7.6} (b) there are natural surjective
group homomorphisms
\begin{eqnarray*}
(\Lambda_1)_{(p)}^{unit} \to ((\Lambda_1)_{(p)}/C_{(p)})^{unit}
\quad\textup{and}\quad 
(\Lambda_2)_{(p)}^{unit} \to ((\Lambda_2)_{(p)}/C_{(p)})^{unit}.
\end{eqnarray*}
The composed group homomorphism
\begin{eqnarray}\label{8.11}
(\Lambda_1)_{(p)}^{unit}\to 
((\Lambda_1)_{(p)}/C_{(p)})^{unit} \to 
\frac{((\Lambda_1)_{(p)}/C_{(p)})^{unit}}
{((\Lambda_2)_{(p)}/C_{(p)})^{unit}}
\end{eqnarray}
is surjective, and the kernel contains
$(\Lambda_2)_{(p)}^{unit}$. 

We want to show that the kernel is precisely 
$(\Lambda_2)_{(p)}^{unit}$. 
Let $a\in (\Lambda_1)_{(p)}^{unit}$ be in the kernel.
We want to show $a\in (\Lambda_2)_{(p)}^{unit}$.

The image of $a$ in $((\Lambda_1)_{(p)}^{unit}/C_{(p)})^{unit}$
is in $((\Lambda_2)_{(p)}/C_{(p)})^{unit}$, 
so by Theorem \ref{t7.6} (b) it has a preimage
$b$ in $(\Lambda_2)_{(p)}^{unit}$. Therefore $ab^{-1}$
has the unit element in $((\Lambda_1)_{(p)}/C_{(p)})^{unit}$
as image in $((\Lambda_1)_{(p)}/C_{(p)})^{unit}$.
Thus $ab^{-1}=1_A+f$ for some $f\in C_{(p)}$. 
So $a=b+bf\in (\Lambda_2)_{(p)}$.
The same argument shows $a^{-1}\in (\Lambda_2)_{(p)}$,
so $a\in (\Lambda_2)_{(p)}^{unit}$.
Therefore the kernel of the composed map in \eqref{8.11} is 
$(\Lambda_2)_{(p)}^{unit}$. 

(d) Consider $a\in\Lambda_1$ with $a+C\in (\Lambda_1/C)^{unit}$.
We claim
\begin{eqnarray}\label{8.12}
(a+C)\Lambda_1=\Lambda_1.
\end{eqnarray}
Proof of the claim: An element $\www{a}\in\Lambda$ with
$(a+C)(\www{a}+C)=1_A+C$ exists. Then
\begin{eqnarray*}
\Lambda_1\supset (a+C)\Lambda_1\supset (a+C)(\www{a}+C)\Lambda_1
=(1_A+C)\Lambda_1=\Lambda_1. \hspace*{1cm}(\Box)
\end{eqnarray*}

By the proof of Theorem \ref{t7.6} (b) for each $p\in P_0$
an element $d_p\in a+C\subset\Lambda_1$ with
$d_p\in (\Lambda_1)_{(p)}^{unit}$ exists.

By the proof of part (c), the image in
$\bigoplus_{p\in P_0}(\Lambda_1)_{(p)}^{unit}/
(\Lambda_2)_{(p)}^{unit}$ of the class 
$((a+C)\mmod (\Lambda_2/C)^{unit})$ under the isomorphism in
\eqref{8.4} is the class of the tuple
$(d_p)_{p\in P_0}\in \prod_{p\in P_0}(\Lambda_1)_{(p)}^{unit}$. 

Define $d_p:=1_A$ for $p\in\P-P_0$. By part (b) the image
of $(d_p)_{p\in P_0}$ in $\ker(G(\Lambda_2)\to G(\Lambda_1))$
is the full lattice $L=\bigcap_{p\in P_0}d_p(\Lambda_2)_{(p)}$.
It remains to show $L=C+a\Lambda_2$. 

For $p\in P_0$ write $a=d_p+f_p$ for some $f_p\in C$. 
For $p\in P_0$ 
\begin{eqnarray*}
(C+a\Lambda_2)_{(p)}
&=& (C+d_p\Lambda_2+f_p\Lambda_2)_{(p)}
=(C+d_p\Lambda_2)_{(p)}\\
&=& C_{(p)} + d_p(\Lambda_2)_{(p)}
=d_pC_{(p)}+d_p(\Lambda_2)_{(p)}\\
&=& d_p(C+\Lambda_2)_{(p)}=d_p(\Lambda_2)_{(p)}.
\end{eqnarray*}
For $p\in\P-P_0$ 
\begin{eqnarray*}
(C+a\Lambda_2)_{(p)}&=& C_{(p)}+a(\Lambda_2)_{(p)}
=C_{(p)}+a(\Lambda_1)_{(p)}
=((C+a)\Lambda_1)_{(p)}\\
&\stackrel{\eqref{8.12}}{=}& (\Lambda_1)_{(p)}=(\Lambda_2)_{(p)}=d_p(\Lambda_2)_{(p)}.
\end{eqnarray*}
Therefore $L=\bigcap_{p\in \P}(C+a\Lambda_2)_{(p)}
=C+a\Lambda_2$. This shows part (d). 

(e) {\bf Claim:} {\it The group homomorphism
$$\ker(G(\Lambda_2)\to G(\Lambda_1))
\to 
\ker(G([\Lambda_2]_\varepsilon)\to G([\Lambda_1]_\varepsilon))$$
is surjective.}

Proof of the Claim: Consider $\www{L}\in G(\Lambda_2)$ with
$[\www{L}]_\varepsilon\in \ker(G([\Lambda_2]_\varepsilon)\to
G([\Lambda_1]_\varepsilon))$. 
Then $[\Lambda_1\www{L}]_\varepsilon=[\Lambda_1]_\varepsilon$,
so an element $a\in A^{unit}$ with 
$\Lambda_1\www{L}=a\Lambda_1$ exists.
Define $L:=a^{-1}\www{L}$. Then
$[L]_\varepsilon=[\www{L}]_\varepsilon$, 
$L\in G(\Lambda_2)$, $\Lambda_1L=\Lambda_1$, so
$L\in\ker (G(\Lambda_2)\to G(\Lambda_1))$.
\hfill ($\Box$)

\medskip
The Claim and the exactness of the sequence in \eqref{8.3} 
show that the part
$$\prod_{p\in P_0}(\Lambda_1)_{(p)}^{unit} / 
(\Lambda_2)_{(p)}^{unit} \to G([\Lambda_2]_\varepsilon)
\to G([\Lambda_1]_\varepsilon) \to 1$$
of the sequence in \eqref{8.7} is exact.

Because of 
$\Lambda_2^{unit}=\bigcap_{p\in \P}(\Lambda_2)_{(p)}^{unit}$
the kernel of the map
\begin{eqnarray}\label{8.13}
\Lambda_1^{unit}\to\prod_{p\in P_0} 
(\Lambda_1)_{(p)}^{unit} / (\Lambda_2)_{(p)}^{unit}
\end{eqnarray}
is $\Lambda_2^{unit}$. It remains to show that the image of the
map in \eqref{8.13} is the kernel of the map
\begin{eqnarray}\label{8.14}
\prod_{p\in P_0}(\Lambda_1)_{(p)}^{unit} / 
(\Lambda_2)_{(p)}^{unit}\to G([\Lambda_2]_\varepsilon).
\end{eqnarray}

First consider $a\in \Lambda_1^{unit}$ and 
$(a_p)_{p\in \P}$ with $a_p:=a$ for $p\in P_0$ 
and $a_p:=1_A$ for $p\in \P-P_0$. 
The image in $G(\Lambda_2)$ of the class in $\prod_{p\in P_0}
(\Lambda_1)_{(p)}^{unit}/(\Lambda_2)_{(p)}^{unit}$ of
the tuple $(a_p)_{p\in P_0}\in
\prod_{p\in P_0}(\Lambda_1)_{(p)}^{unit}$ is
\begin{eqnarray*}
L&:=& \bigcap_{p\in\P} a_p(\Lambda_2)_{(p)}
=\bigcap_{p\in P_0}a(\Lambda_2)_{(p)}\cap 
\bigcap_{p\in\P-P_0}(\Lambda_2)_{(p)}\\
&=& \bigcap_{p\in\P} a(\Lambda_2)_{(p)}
=a\Lambda_2\in [\Lambda_2]_\varepsilon.
\end{eqnarray*}
Therefore the image of the map in \eqref{8.13} is contained in the 
kernel of the map in \eqref{8.14}.

Finally consider a tuple 
$(a_p)_{p\in P_0}\in\prod_{p\in P_0}(\Lambda_1)_{(p)}^{unit}$
whose class in $\prod_{p\in P_0} (\Lambda_1)_{(p)}^{unit} / 
(\Lambda_2)_{(p)}^{unit}$ is in the kernel of the map 
to $G([\Lambda_2]_\varepsilon)$ in \eqref{8.14}.
Define $a_p:=1_A$ for $p\in \P-P_0$. Then the full lattice
$L$ with $L_{(p)}=a_p(\Lambda_2)_{(p)}$ for each $p\in\P$
satisfies $[L]_\varepsilon=[\Lambda_2]_\varepsilon$.
Therefore $L=b\Lambda_2$ for some $b\in A^{unit}$.
Comparison with $L_{(p)}=a_p(\Lambda_2)_{(p)}$ gives
$$\frac{b}{a_p}\in (\Lambda_2)_{(p)}^{unit}\subset (\Lambda_1)_{(p)}^{unit}
\quad\textup{for each }p\in\P,$$
so $b\in \bigcap_{p\in \P}(\Lambda_1)_{(p)}^{unit}=\Lambda_1^{unit}$.
Therefore the image of the map in \eqref{8.13} is the kernel
of the map in \eqref{8.14}.

(f) The left hand side of the isomorphism in part (c) is
obviously finite.
One applies the exact sequence in part (e) and the isomorphism
in part (c). Especially, the group $\Lambda_2^{unit}$ has
finite index in the group $\Lambda_1^{unit}$. 
\hfill$\Box$

\begin{remarks}\label{t8.3}
Comparing Theorem \ref{t8.2} with \cite[Ch. I \S 12]{Ne99},
this reference has the following restrictions and differences.
There $A$ is an algebraic number field,
$\Lambda_1$ is the maximal order $\Lambda_{max}$ in $A$,
and localization is done by prime ideals in $\Lambda_2$.
Then the Propositions (12.9) and (12.11) in 
\cite[Ch. I \S 12]{Ne99} are analogs of the parts
(e) and (c) of Theorem \ref{t8.2}. Theorem (12.12) in 
\cite[Ch. I \S 12]{Ne99} is a special case of part (f)
of Theorem \ref{t8.2}. 

The surjectivity of the map in \eqref{8.1} for the case
when $A$ is an algebraic number field is also proved in
\cite[Corollary 2.1.11]{DTZ62}.
\end{remarks}

\section{Dividing out the radical}\label{s9}
\setcounter{equation}{0}
\setcounter{table}{0}

Throughout the whole paper $A$ is as in Theorem \ref{t3.1}, 
so $A$ is a finite dimensional commutative $\Q$-algebra 
with unit element $1_A$. We use the notations in Theorem \ref{t3.1}.
Recall that the radical $R=\bigoplus_{j=1}^k N^{(j)}$ is an 
ideal, and all its elements are nilpotent.
Recall the separable subalgebra 
$F:=\bigoplus_{j=1}^k F^{(j)}\subset A$ of $A$, the
natural decomposition $A=F\oplus R$, and the induced
projection $\pr_F:A\to F$. By Remark \ref{t3.2} (ii)
it respects addition, multiplication and division in $A$
and $F$.

This section is devoted to the action of the projection 
$\pr_F:A\to F$ on invertible full lattices.
First Lemma \ref{t9.1} states basic facts.

\begin{lemma}\label{t9.1}
(a) The image $\pr_FL$ of a full lattice in $A$ is a full
lattice in $F$. The projection $\pr_F:A\to F$ is compatible 
with the multiplication in $\LL(A)$ and $\LL(F)$, 
\begin{eqnarray}
\pr_F(L_1\cdot L_2)=\pr_F L_1\cdot \pr_F L_2\quad
\textup{for }L_1,L_2\in\LL(A).\label{9.1}
\end{eqnarray}
It induces a surjective homomorphism 
$\pr_F:\LL(A)\to\LL(F)$ of semigroups. 
Especially, it maps idempotents to idempotents,
so orders to orders. So it restricts to a surjective
homomorphism
\begin{eqnarray}\label{9.2}
\pr_F:\{\textup{orders in }A\}\to\{\textup{orders in }F\}
\end{eqnarray}
of semigroups. 
Though it does not respect the division maps. In general 
we only have 
\begin{eqnarray}\label{9.3}
\pr_F(L_1: L_2)\subset\pr_F L_1:\pr_FL_2\quad
\textup{for }L_1,L_2\in\LL(A).
\end{eqnarray}

(b) The projection $\pr_F:\LL(A)\to\LL(F)$ respects
$\varepsilon$-equivalence and $w$-equivalence. Therefore it 
induces surjective homomorphisms
\begin{eqnarray}\label{9.4}
\pr_F:\EE(A)&\to& \EE(F),\\
\pr_F:W(\LL(A))&\to& W(\LL(F))\label{9.5}
\end{eqnarray}
of semigroups. 

(c) In general, for $L\in\LL(A)$ we only have an inclusion
$\pr_F\OO(L)\subset\OO(\pr_FL)$. But if $L$ is invertible
then $\pr_FL$ is invertible with 
$\OO(\pr_FL)=\pr_F(\OO(L))$ and $(\pr_FL)^{-1}=\pr_F(L^{-1})$.

(d) Let $\Lambda$ be an order in $L$. Write 
$\Lambda_0:=\pr_F\Lambda$ for the induced order in $F$.
Part (c) gives rise to two group homomorphisms,
\begin{eqnarray}\label{9.6}
G(\Lambda)&\to& G(\Lambda_0),\quad L\mapsto \pr_FL,\\
G([\Lambda]_\varepsilon)&\to& G([\Lambda_0]_\varepsilon),\quad
[L]_\varepsilon\mapsto [\pr_FL]_\varepsilon.\label{9.7}
\end{eqnarray}
\end{lemma}

{\bf Proof:}
The parts (a) and (b) are obvious. The inclusion
$\pr_F\OO(L)\subset \OO(\pr_FL)$ follows from \eqref{9.3}.
Let $L\in\OO(L)$ be invertible. Then by \eqref{9.1}
$$\pr_FL\cdot \pr_FL^{-1}=\pr_F(L\cdot L^{-1})=\pr_F(\OO(L)),$$
and therefore $\pr_FL$ is invertible with 
$\OO(\pr_FL)=\pr_F\OO(L)$ and $(\pr_FL)^{-1}=\pr_FL^{-1}$.
Now part (d) is obvious.\hfill$\Box$

\bigskip
The group homomorphisms in \eqref{9.6} and \eqref{9.7}
have surprisingly good properties. 
The more difficult part of Theorem \ref{t9.2} is due to 
Faddeev \cite[Theorem 3]{Fa68}, namely the injectivity of the
map in \eqref{9.7}.

\begin{theorem}\label{t9.2}
The group homomorphism in \eqref{9.6} is surjective.
The group homomorphism in \eqref{9.7} is an isomorphism.
\end{theorem}

{\bf Proof:} We show that the group homomorphism in \eqref{9.6}
is surjective. It implies immediately that also the group
homomorphism in \eqref{9.7} is surjective.
Together with Faddeev's result \cite[Theorem 3]{Fa68} that
the group homomorphism in \eqref{9.7} is injective, this
implies that it is an isomorphism.

Let $L_0\in G(\Lambda_0)$, so $L_0\in\LL(F)$ is invertible with
$\OO(L_0)=\Lambda_0$. By Theorem \ref{t7.3} there is an element
$a_p\in F^{unit}$ for each $p\in\P$ with 
$$(L_0)_{(p)}=a_p(\Lambda_0)_{(p)},$$
and there is a finite set $P_0\subset\P$ such that we can
choose $a_p=1$ for $p\in \P-P_0$. 

Of course $F^{unit}\subset A^{unit}$. By Theorem \ref{t7.2} (c) 
there is a unique full lattice $L$ in $A$ with
$$L_{(p)}=a_p\Lambda_{(p)}\quad\textup{for each }p\in\P.$$
By Theorem \ref{t7.3} it is invertible with $\OO(L)=\Lambda$.
Of course $\pr_F(L)=L_0$.

This shows the surjectivity of the maps in \eqref{9.6} and
\eqref{9.7}. The map in \eqref{9.7} is injective by 
\cite[Theorem 3]{Fa68}.
\hfill$\Box$

\begin{remarks}\label{t9.3}
(i) Also 
Faddeev's proof of the injectivity of the map in \eqref{9.7}
uses the localization in section \ref{s7}. It is not difficult.
We have a different proof without localization.

(ii) The proof above of the surjectivity of the map
in \eqref{9.6} uses the localization in section \ref{s7}.
It is simple. Though in this case we do not have a different
proof without localization. It seems difficult to see
the surjectivity without the tools from section \ref{s7}. 

(iii) The second group $G([\Lambda_0]_\varepsilon)$ in 
\eqref{9.7} is finite by the Jordan-Zassenhaus Theorem
\ref{t6.3}. Therefore also the first group
$G([\Lambda]_\varepsilon)$ in \eqref{9.7} is finite.
We know this already from Theorem \ref{t6.5}.
But the equality of the sizes of the two groups is new.

(iv) By \eqref{5.13} any $w$-equivalence class of 
$\varepsilon$-classes $[L]_\varepsilon$ of full lattices
$L$ with $\OO(L)=\Lambda$ is in bijection to 
$G([\Lambda]_\varepsilon)$, so finite. 
Therefore we would obtain a second proof of Theorem \ref{t6.5}
if we would have an independent proof that for each order
$\Lambda$ in $A$ the number of $w$-equivalence classes
with this order is finite.

(v) Suppose $A\supsetneqq F$, so $A$ is not separable.
Let $\Lambda_0$ be an order in $F$. By Theorem \ref{t6.4}
there are infinitely many orders $\Lambda$ in $A$
with $\pr_F\Lambda=\Lambda_0$. By Theorem \ref{t9.2}
the finite groups $G([\Lambda]_\varepsilon)$ for all these
orders are canonically isomorphic to 
$G([\Lambda_0)]_\varepsilon)$ and thus also canonically
isomorphic to one another.
If $\Lambda$ and $\www{\Lambda}$ are two such orders
with $\Lambda\supset\www{\Lambda}$ then the surjective
group homomorphism $G([\www{\Lambda}]_\varepsilon)\to
G([\Lambda]_\varepsilon)$ from Theorem \ref{t8.2} is
the isomorphism induced by 
$G([\www{\Lambda}]_\varepsilon)\cong
G([\Lambda_0]_\varepsilon)\cong G([\Lambda]_\varepsilon)$.

(vi) But in the case $A\supsetneqq F$, for two orders
$\Lambda$ and $\www{\Lambda}$ in $A$ with
$\pr_F\Lambda=\pr_F\www{\Lambda}$, the finite sets
$\{[L]_\varepsilon\,|\, L\in\LL(A),\OO(L)=\Lambda\}$ and
$\{[L]_\varepsilon\,|\, L\in\LL(A),\OO(L)=\www{\Lambda}\}$
may have very different sizes. 
\end{remarks}

\section{Sufficiently high powers of full lattices
are invertible}\label{s10}
\setcounter{equation}{0}
\setcounter{table}{0}

A main result in \cite{DTZ62} is Theorem C in section 1.5.
It says that if $A$ is an algebraic number field of dimension
$n\in\Z_{\geq 2}$ then for each full lattice $L\in\LL(A)$
each power $L^k$ with $k\geq n-1$ is invertible.
This result was generalized in \cite[Theorem 2]{Si70}
to the case when $A$ is separable, so a direct sum
of algebraic number fields. Here we generalize it further to
the case of our standard situation.

\begin{theorem}\label{t10.1}
Let $A$ be a commutative $\Q$-algebra of dimension 
$n\in\Z_{\geq 2}$ with unit element $1_A$. For each full lattice
$L\in\LL(A)$ each power $L^k$ with $k\geq n-1$ is invertible.
\end{theorem}

Our proof follows roughly the proof in \cite[2.2]{DTZ62}.
But there are differences. The proof in \cite[2.2]{DTZ62}
covers situations which generalize the case of an algebraic
number field $A$ in a way which is not relevant for us.
Theorem C in \cite[1.5]{DTZ62} is accompanied by Theorem A
and Theorem B which do not apply to our situation.
The proof in \cite[2.2]{DTZ62} uses localization by prime ideals.
Our proof uses the localization $\Z\dashrightarrow\Z_{(p)}$ 
in section \ref{s7}.

An important first step is Theorem \ref{t10.2} which is the
analogue of \cite[2.2.2 Proposition]{DTZ62} and which is of 
independent interest. In the case of an algebraic number field
Theorem \ref{t10.2} and \cite[2.2.2 Proposition]{DTZ62}
become trivial. Then one can choose $\Lambda=\Lambda_{max}$.
Our proof of Theorem \ref{t10.2} is completely different from
the proof of \cite[2.2.2 Proposition]{DTZ62}.

The further steps in the proof of Theorem \ref{t10.1}
are similar to the steps in \cite[2.2]{DTZ62}.
Most of their statements are collected in Theorem \ref{t10.3}.
We prove first Theorem \ref{t10.2}, then Theorem \ref{t10.3}
and finally Theorem \ref{t10.1}.

\begin{theorem}\label{t10.2}
Let $A$ be a finite dimensional commutative $\Q$-algebra 
with unit element $1_A$. Let $L\in\LL(A)$ be a full lattice.
An order $\Lambda\supset\OO(L)$ with $\Lambda L\in G(\Lambda)$
(so $\OO(\Lambda L)=\Lambda$ and $\Lambda L$ invertible)
exists.
\end{theorem}

{\bf Proof:} 
We will first construct an order $\Lambda$ and then show that
it works. Recall the decomposition $A=\bigoplus_{j=1}^kA^{(j)}$
for some $k\in\N$ with $A^{(j)}$ as in (i)--(iii) in 
Theorem \ref{t3.1} and especially 
$A^{(j)}=F^{(j)}\oplus N^{(j)}$ with $F^{(j)}$ an algebraic
number field and $N^{(j)}$ the unique maximal ideal in $A^{(j)}$.

Define $n_j:=\max(n\in\Z_{\geq 0}\,|\, (N^{(j)})^n\neq\{0\})$
for $j\in\{1,...,k\}$. 
The subspaces of the decreasing filtration
\begin{eqnarray*}
A^{(j)}=(N^{(j)})^0\supset N^{(j)}\supset (N^{(j)})^2\supset ...\supset
(N^{(j)})^{n_j}\supset (N^{(j)})^{n_j+1}=\{0\}
\end{eqnarray*}
are $F^{(j)}$-vector spaces. Choose a splitting
\begin{eqnarray}\label{10.1}
A^{(j)}=B^{(j,0)}\oplus B^{(j,1)}\oplus ...\oplus B^{(j,n_j)}
\end{eqnarray}
into $F^{(j)}$-vector spaces with 
$B^{(j,0)}=F^{(j)}$ which splits this filtration,
so with
\begin{eqnarray}\label{10.2}
(N^{(j)})^l=B^{(j,l)}\oplus (N^{(j)})^{l+1}
\quad\textup{for }l\in\{0,1,...,n_j\}.
\end{eqnarray}
Then
\begin{eqnarray}\label{10.3}
B^{(j,l_1)}\cdot B^{(j,l_2)}\subset\bigoplus_{l\geq l_1+l_2}
B^{(j,l)}
\quad\textup{and}\quad
B^{(j,0)}\cdot B^{(j,l)}=B^{(j,l)}.
\end{eqnarray}
Let $\pr_{(j,l)}:A\to B^{(j,l)}$ be the projection with respect
to the splitting $A=\bigoplus_{j=1}^k\bigoplus_{l=0}^{n_j}
B^{(j,l)}$. Consider the full lattice
\begin{eqnarray}\label{10.4}
L^{(j,L)}:=\pr_{(j,l)}(L)\in\LL(B^{(j,l)})
\end{eqnarray}
in $B^{(j,l)}$. 

We can and will choose for each  
$(j,l)\in\{1,...,k\}\times\{0,1,...,n_j\}$ a full lattice 
$B^{(j,l)}_\Z\in \LL(B^{(j,l)})$ in $B^{(j,l)}$ such that the
full lattices for fixed $j$ have the following properties:
\begin{eqnarray}\label{10.5}
B^{(j,0)}_\Z&=& \Lambda_{max}(F^{(j)}),\\
B^{(j,l)}_\Z L^{(j,0)}&\supset& L^{(j,l)}\label{10.6},\\
B^{(j,l)}_\Z&\supset& \pr_{(j,l)}(\OO(L))\label{10.7},\\
B^{(j,l_1)}_\Z B^{(j,l_2)}_\Z &\subset&
\sum_{l\geq l_1+l_2}B^{(j,l)}_\Z \label{10.8}.
\end{eqnarray}
In fact, \eqref{10.8} contains
\begin{eqnarray}\label{10.9}
B^{(j,0)}_\Z B^{(j,l)}_\Z=B^{(j,l)}_\Z\quad\textup{for }l\geq 0.
\end{eqnarray}
(here $\supset$ follows from $1_{A^{(j)}}\in B^{(j,0)}_\Z$).
Observe $q\cdot 1_{A^{(j)}}\in L^{(j,0)}$ for some 
$q\in\Q-\{0\}$ and thus 
$B^{(j,l)}_\Z L^{(j,0)}\supset qB^{(j,l)}_\Z$.
The full lattices $B^{(j,l)}_\Z$ are chosen with increasing
$l$, so after $B^{(j,0)}_\Z$ first $B^{(j,1)}_\Z$, second 
$B^{(j,2)}_\Z$ and last $B^{(j,n_j)}_\Z$. 

Define the full lattice in $A$
\begin{eqnarray}\label{10.10}
\Lambda := \bigoplus_{j=1}^k\Lambda^{(j)}\quad\textup{with}\quad
\Lambda^{(j)}:=\bigoplus_{l=0}^{n_j}B^{(j,l)}_\Z.
\end{eqnarray}
$\Lambda^{(j)}$ is an order in $A^{(j)}$ 
because of \eqref{10.5} and \eqref{10.8},
which give $1_A\in\Lambda^{(j)}$ and 
$\Lambda^{(j)}\cdot\Lambda^{(j)}\subset\Lambda^{(j)}$.
Therefore $\Lambda$ is an order in $A$.  
It contains $\OO(L)$ because of \eqref{10.7}.

\medskip
{\bf Claims:} 
\begin{list}{}{}
\item[(i)] $\Lambda\bigoplus_{j=1}^kL^{(j,0)}$
is invertible with order $\OO(\Lambda\bigoplus_{j=1}^kL^{(j,0)})
=\Lambda$.
\item[(ii)]  $\Lambda\bigoplus_{j=1}^kL^{(j,0)}=\Lambda L$.
\end{list}

\medskip
Together the Claims (i) and (ii) give Theorem \ref{t10.2}.

{\bf Proof of Claim (i):}
The order of the full lattice $B^{(j,0)}_\Z L^{(j,0)}$ in 
$F^{(j)}$ is the maximal order 
$B^{(j,0)}_\Z=\Lambda_{max}(F^{(j)})$ in $F^{(j)}$. 
Therefore $B^{(j,0)}_\Z L^{(j,0)}$ is invertible with
$$\bigl(B^{(j,0)}_\Z L^{(j,0)}\bigr)
\bigl(B^{(j,0)}_\Z L^{(j,0)}\bigr)^{-1} =B^{(j,0)}_\Z.$$
Therefore 
$$\bigl(\Lambda^{(j)}L^{(j,0)}\bigr)
\bigl( \Lambda^{(j)}(B^{(j,0)}_\Z L^{(j,0)})^{-1}\bigr)
=\Lambda^{(j)}B^{(j,0)}_\Z=\Lambda^{(j)},$$
so the full lattice $\Lambda^{(j)}L^{(j,0)}$ in $A^{(j)}$ is
invertible with order 
$\OO(\Lambda^{(j)}L^{(j,0)})=\Lambda^{(j)}$. 
Thus the full lattice
$$\Lambda\bigoplus_{j=1}^kL^{(j,0)}
=\bigoplus_{j=1}^k \Lambda^{(j)}L^{(j,0)}$$
in $A$ is invertible with order $\bigoplus_{j=1}^k\Lambda^{(j)}
=\Lambda$.
\hfill$(\Box)$

{\bf Proof of Claim (ii):}
The inclusion $\supset$: 
\begin{eqnarray*}
\Lambda L&\subset & 
\Lambda\bigoplus_{j=1}^k\bigoplus_{l=0}^{n_j}L^{(j,l)}
=\bigoplus_{j=1}^k\Bigl(\Lambda^{(j)}\bigoplus_{l=0}^{n_j}
L^{(j,l)}\Bigr)\\
&\stackrel{\eqref{10.6}}{\subset}& 
\bigoplus_{j=1}^k \Bigl(\Lambda^{(j)}\bigoplus_{l=0}^{n_j}
B^{(j,l)}_\Z L^{(j,0)}\Bigr)
= \bigoplus_{j=1}^k \Lambda^{(j)}L^{(j,0)}
=\Lambda\bigoplus_{j=1}^kL^{(j,0)}.
\end{eqnarray*}

The inclusion $\subset$: 
We will show for $j\in\{1,...,k\}$
\begin{eqnarray}\label{10.11}
B^{(j,n_j-l-m)}_\Z L^{(j,l)}\subset \Lambda^{(j)}L
\quad\textup{for }l\in\{0,1,...,n_j\}\textup{ and }m\in\Z
\end{eqnarray}
by induction in $m$. First
\begin{eqnarray*}
B^{(j,n_j-l-m)}_\Z L^{(j,l)}
&\stackrel{\eqref{10.6}}{\subset}& 
B^{(j,n_j-l-m)}_\Z B^{(j,l)}_\Z L^{(j,0)}\\
&\stackrel{\eqref{10.8}}{\subset}&
\Bigl(\sum_{r\geq n_j-m}B^{(j,r)}_\Z\Bigr)L^{(j,0)}\\
&=& \left\{\begin{array}{ll}
\{0\}\cdot L^{(j,0)}=\{0\}&\textup{ if }m<0,\\
B^{(j,n_j)}_\Z L^{(j,0)}= B^{(j,n_j)}_\Z L
\subset \Lambda^{(j)}L&\textup{ if }m=0.
\end{array}\right.
\end{eqnarray*}
So \eqref{10.11} holds for $m\leq 0$. 

Suppose \eqref{10.11} holds for some $m\geq 0$ and each
$l\in\{0,1,...,n_j\}$. Then by the calculation above and by
induction hypothesis
\begin{eqnarray*}
B^{(j,n_j-l-m-1)}_\Z L^{(j,l)}
&\subset& \Bigl(\sum_{r\geq n_j-m-1}B^{(j,r)}_\Z\Bigr)L^{(j,0)}\\
&\subset& B^{(j,n_j-m-1)}_\Z L^{(j,0)} + \Lambda^{(j)}L\\
&\subset& B^{(j,n_j-m-1)}_\Z\bigl(L+\sum_{s\geq 1}L^{(j,s)}\bigr)
+\Lambda^{(j)}L\\
&=& B^{(j,n_j-m-1)}_\Z L + \Lambda^{(j)}L\\
&=& \Lambda^{(j)}L.
\end{eqnarray*}
\eqref{10.11} is proved. 
The case $(l,m)=(0,n_j)$ gives 
$B^{(j,0)}_\Z L^{(j,0)}\subset \Lambda^{(j)}L$, so 
$L^{(j,0)}\subset \Lambda^{(j)}L\subset\Lambda L$,
so $\Lambda\bigoplus_{j=1}^kL^{(j,0)}\subset\Lambda L$.
This concludes the proof of Claim (ii) and of Theorem 
\ref{t10.2}.\hfill$\Box$

\begin{theorem}\label{t10.3}
Let $A$ be a commutative $\Q$-algebra with unit element $1_A$ of 
dimension $n\in\Z_{\geq 2}$. Let $L\in\LL(A)$ be a full lattice.
Let $\Lambda$ be an order with $\Lambda\supset \OO(L)$
and $\Lambda L\in G(\Lambda)$. 

(a) A full lattice $L_1\in\LL(A)$ with $L_1\sim_w L$
and $L_1\Lambda=\Lambda$ exists.

(b) A full lattice $L_2\in\LL(A)$ with $L_2\sim_w L$,
$1_A\in L_2$ and $L_2\subset \Lambda$ exists.

(c) The full lattice $L_2$ in part (b) satisfies:
\begin{list}{}{}
\item[(i)] The sequence $(L_2^l)_{l\in\N}$ of full lattices
is increasing and becomes stationary, so there is a minimal
number $N\in\N$ with $L_2^N=L_2^{N+l}$ for each $l\geq 0$.
\item[(ii)] $\Lambda_2:=L_2^N$ is an order with 
$\Lambda_2\supset\OO(L)$.
\end{list}

(d) In part (c) $N\leq n-1$.
\end{theorem}

{\bf Proof:}
(a) By Theorem \ref{t8.2} (b) the map
$$G(\OO(L))\to G(\Lambda),\quad K\mapsto \Lambda K,$$
is a surjective group homomorphism.
Choose $K\in G(\OO(L))$ with $\Lambda K=\Lambda L$.
Then $L_1:=LK^{-1}$ satisfies $L_1\Lambda=\Lambda$ and
$L_1\sim_w L$, the last statement because of Theorem \ref{t5.8}.
Especially $L_1\subset\Lambda$ because $1_A\in\Lambda$. 

(b) For almost all prime numbers $q$ $(L_1)_{(q)}=\Lambda_{(q)}$.
Let $P_0\subset \P$ be the finite set of prime numbers $p$
with $(L_1)_{(p)}\subsetneqq \Lambda_{(p)}$.

Fix a prime number $p\in P_0$. As in the proof of Theorem 
\ref{t7.6}, by Remark \ref{t3.2} (i), 
Theorem \ref{t3.1} applies to $\Lambda/p\Lambda$
instead of $A$ and $\F_p$ instead of $\Q$.
We write the decompositions in Theorem \ref{t3.1} of
$\Lambda/p\Lambda$ as
\begin{eqnarray}\label{10.12}
\Lambda/p\Lambda &=& 
\bigoplus_{j=1}^{\www{k}} A_{\Lambda,p}^{(j)},\qquad
A_{\Lambda,p}^{(j)}\ =\ F_{\Lambda,p}^{(j)}\oplus
N_{\Lambda,p}^{(j)}.
\end{eqnarray}
Here $\www{k}\in\N$, 
$A_{\Lambda,p}^{(1)},...,A_{\Lambda,p}^{\www{k}}$ are 
irreducible and local $\F_p$-algebras with
$A_{\Lambda,p}^{(i)}\cdot A_{\Lambda,p}^{(j)}=\{0\}$ for
$i\neq j$. $N_{\Lambda,p}^{(j)}$ is the maximal ideal
in $A_{\Lambda,p}^{(j)}$ and consists of nilpotent elements.
$F_{\Lambda,p}^{(j)}$ is an $\F_p$-subalgebra of
$A_{\Lambda,p}^{(j)}$ and a finite field extension of $\F_p$.
$A_{\Lambda,p}^{(j)}$ is an $F_{\Lambda,p}^{(j)}$-algebra.

An element $b=\sum_{j=1}^{\www{k}}(c_j+n_j)\in \Lambda/p\Lambda$
with $c_j\in F_{\Lambda,p}^{(j)}$ and 
$n_j\in N_{\Lambda,p}^{(j)}$ is a unit if and only
if each $c_j\neq 0$. Therefore the set of elements
in $\Lambda/p\Lambda$ which are not units is
\begin{eqnarray*}
\Lambda/p\Lambda-(\Lambda/p\Lambda)^{unit}
=\bigcup_{j=1}^{\www{k}}\Bigl(N_{\Lambda,p}^{(j)}\oplus
\bigoplus_{l\neq j}A_{\Lambda,p}^{(l)}\Bigr),
\end{eqnarray*}
so it is a union of $\www{k}$ proper $\Lambda/p\Lambda$-ideals.

The $\F_p$-subspace $(L_1+p\Lambda)/p\Lambda$ is not contained
in any one of these $\Lambda/p\Lambda$-ideals because
$(L_1+p\Lambda)/p\Lambda\cdot\Lambda/p\Lambda=\Lambda/p\Lambda$.
So it is also not contained in their union.

Therefore we can choose an element $b_p\in L_1\subset\Lambda$ 
with
\begin{eqnarray*}
[b_p]=b_p+p\Lambda\in (L_1+p\Lambda)/p\Lambda\cap 
(\Lambda/p\Lambda)^{unit}.
\end{eqnarray*}
By Lemma \ref{t7.5} (b) $b_p\in \Lambda_{(p)}^{unit}$. 

For $q\in\P-P_0$ define $b_q:=1_A$. By Theorem \ref{t7.3}
there is a unique full lattice $L_2\in\LL(A)$ with
\begin{eqnarray*}
(L_2)_{(p)} = b_p^{-1}(L_1)_{(p)}\quad\textup{for each }
p\in\P.
\end{eqnarray*}
By Remark \ref{t7.4}
$$L_2\sim_w L_1\sim_w L.$$
By construction $1_A= b_p^{-1}b_p\in (L_2)_{(p)}$ for each
$p\in\P$, so $1_A\in L_2$. 
For each $p\in\P$ $(L_2)_{(p)}\subset \Lambda_{(p)}$ because
$b_p\in \Lambda_{(p)}^{unit}$, so $L_2\subset\Lambda$.

(c) (i) The sequence $(L_2^l)_{l\in\N}$ is increasing because
of $1_A\in L_2$. It becomes stationary at some $L_2^N$ because
each $L_2^l\subset \Lambda$.

(ii) $\Lambda_2:=L_2^N$ is an order because $1_A\in L_2^N$
and $L_2^N\cdot L_2^N=L_2^N$.

(d) For $q\in \P-P_0$ 
$(L_2)_{(q)}=(L_1)_{(q)}=\Lambda_{(q)}=(\Lambda_2)_{(q)}$.

Fix a prime number $p\in P_0$. Consider the $n$-dimensional
$\F_p$-vector space $\Lambda_2/p\Lambda_2$ and the sequence
$((L_2^l+p\Lambda_2)/p\Lambda_2)_{l\in\N}$ of increasing
$\F_p$-subspaces. Then 
$(L_2^l+p\Lambda_2)/p\Lambda_2=\Lambda_2/p\Lambda_2$
for $l\geq N$. On the other hand, if for some $m\in\N$ 
\begin{eqnarray*}
(L_2^m+p\Lambda_2)/p\Lambda_2 
&=& (L_2^{m+1}+p\Lambda_2)/p\Lambda_2,\\
\textup{then}\quad 
(L_2^m+p\Lambda_2)/p\Lambda_2 
&=& (L_2^{m+l}+p\Lambda_2)/p\Lambda_2
\quad\textup{for each }l\geq 1.
\end{eqnarray*}
Therefore there is a number $N_p\leq N$ with
\begin{eqnarray*}
\frac{L_2+p\Lambda_2}{p\Lambda_2}\subsetneqq
\frac{L_2^2+p\Lambda_2}{p\Lambda_2}\subsetneqq
...\subsetneqq
\frac{L_2^{N_p}+p\Lambda_2}{p\Lambda_2}
=\frac{L_2^{N_p+l}+p\Lambda_2}{p\Lambda_2}
=\frac{\Lambda_2}{p\Lambda_2}\textup{ for }l\geq 0.
\end{eqnarray*}
$(L_2+p\Lambda_2)/p\Lambda_2$ has $\F_p$-dimension at least 2,
because else
$$\frac{L_2+p\Lambda_2}{p\Lambda_2}=\F_p[1_A]\quad\textup{and thus}
\quad \F_p[1_A]=\frac{L_2^{N_p}+p\Lambda_2}{p\Lambda_2}
=\frac{\Lambda_2}{p\Lambda_2},$$
a contradiction. Therefore $N_p\leq n-1$. We obtain
$$L_2^{N_p+l}+p\Lambda_2=\Lambda_2\textup{ for }l\geq 0$$
and by the Lemma of Nakayama
$$((L_2)_{(p)})^{N_p+l}=(\Lambda_2)_{(p)}\textup{ for }l\geq 0.$$
The number 
$$N_0:=\max_{p\in P_0}N_p \leq n-1$$
satisfies
\begin{eqnarray*}
((L_2)_{(q)})^{N_0+l} &=& (\Lambda_2)_{(q)}\quad\textup{for }l\geq 0
\textup{ and each }q\in\P,\\
\textup{so}\quad L_2^{N_0+l}&=& \Lambda_2\quad\textup{for }l\geq 0.
\hspace*{2cm}\Box
\end{eqnarray*}

\bigskip
{\bf Proof of Theorem \ref{t10.1}:}
Let $L\in \LL(A)$. By Theorem \ref{t10.2} an order 
$\Lambda\supset\OO(L)$ with $\Lambda L\in G(\Lambda)$ exists.
Choose such an order.  By Theorem \ref{t10.3} a full lattice
$L_2\in\LL(A)$ and an order $\Lambda_2$ with
$L_2\sim_w L$ and $L_2^{n-1+l}=\Lambda_2$ for $l\geq 0$ exist.
By \eqref{5.3} and Theorem \ref{t5.7} (b) 
$\Lambda_2\supset\OO(L_2)=\OO(L)$. Then
by \eqref{4.4}
$$L^{n-1+l}\sim_w L_2^{n-1+l}=\Lambda_2\quad\textup{for }l\geq 0,
$$
so by Theorem \ref{t4.4} (c) and Theorem \ref{t5.7} (b)
$L^{n-1+l}$ is invertible for $l\geq 0$ with
$\OO(L^{n-1+l})=\Lambda_2$. \hfill$\Box$

\end{document}